\documentclass[letterpaper,10pt]{IEEEtran}
\IEEEoverridecommandlockouts
\usepackage{amsmath,graphicx,epsfig,color,amsfonts, algorithm, subfigure}

\graphicspath{{./fig/}}

\usepackage{version,xspace}
\usepackage[noend]{algorithmic}

\def\prob{\mathbb{P}}

\def\expt{\mathbb{E}}
\def\real{\mathbb{R}}

\newcommand{\until}[1]{\{1,\dots, #1\}}

\newcommand{\subscr}[2]{#1_{\textup{#2}}}
\newcommand{\supscr}[2]{#1^{\textup{#2}}}
\newcommand{\setdef}[2]{\{#1 \; | \; #2\}}

\newcommand{\map}[3]{#1: #2 \rightarrow #3}

\newcommand\oprocendsymbol{\hbox{$\square$}}
\newcommand\oprocend{\relax\ifmmode\else\unskip\hfill\fi\oprocendsymbol}

\newcommand\bit[1]{\textit{\textbf{#1}}}

\def \bs {\boldsymbol}
\def \mc {\mathcal}

\def \etal {\emph{et al.}}

\def \mr {\mathrm}

\newtheorem{theorem}{Theorem}
\newtheorem{proposition}[theorem]{Proposition}
\newtheorem{lemma}[theorem]{Lemma}

\newtheorem{remark}{Remark}

\newtheorem{assumption}{Assumption}
\newtheorem{definition}{Definition}

\title{Collective Decision-Making in Ideal Networks:\\
 The Speed-Accuracy Tradeoff
\thanks{This research has been supported in part by ONR grant N00014-09-1-1074 and ARO grant W911NG-11-1-0385.}
\thanks{An earlier version of this work~\cite{VS-NEL:13a} entitled "On the Speed-Accuracy Tradeoff in Collective Decision Making" was presented at the 2013 IEEE Conference on Decision and Control. This work improves the work in~\cite{VS-NEL:13a} and extends it to more general decision-making scenarios.}
}

\author{Vaibhav~Srivastava,~\IEEEmembership{Member,~IEEE}~and~Naomi Ehrich Leonard,~\IEEEmembership{Fellow,~IEEE}
\thanks{V. Srivastava and N. E. Leonard are with the  Mechanical and Aerospace Engineering Department, Princeton University, Princeton, NJ, USA,\tt{ \{vaibhavs, naomi\}@princeton.edu}.}}

\begin{document}

\maketitle

\begin{abstract}
We study collective decision-making in a model of  human groups, with network interactions, performing  two alternative choice tasks. 
{We focus on the speed-accuracy tradeoff, i.e., the tradeoff between a quick decision and a reliable decision, for individuals in the network.}
We model the evidence aggregation process across the network using a coupled drift diffusion model (DDM) and consider the free response paradigm in which individuals take their time to make the decision. 
We develop reduced DDMs as decoupled approximations to the coupled DDM and characterize their efficiency. 
We determine  high probability bounds on the error rate and the expected decision time for the reduced DDM. 
We show the effect of the decision-maker's location in the network on their decision-making performance under several threshold selection criteria. Finally, we extend the coupled DDM to the coupled Ornstein-Uhlenbeck model for decision-making in two alternative choice tasks with recency effects, and to the coupled race model for decision-making in multiple alternative choice tasks. 
\end{abstract}

{
}

\section{Introduction}
Collective cognition and decision-making in human and animal groups have received significant attention 
in a broad scientific community~\cite{IDC:09,LC-CL:09,RDS-CJH-RW:01}. 
Extensive research has led to several models for information assimilation in social networks~\cite{DE-JK:10, MOJ:10}.  Efficient models for decision-making dynamics of a single individual have also been developed~\cite{RB-EB-etal:06,RR:78,RR-GM:08}.
However,  applications like  the deployment of a team of human operators that supervises the operation of automata in complex and uncertain environments involve joint evolution of information assimilation and decision dynamics across the group and the possibility of a \emph{collective intelligence}. 
A principled approach to 
 modeling and analysis of such \emph{socio-cognitive networks} is fundamental to understanding team performance.

In this paper, we focus on the speed-accuracy tradeoff in collective decision-making primarily using the context of problems in which the decision-maker must choose between two alternatives. The speed-accuracy tradeoff is the fundamental tradeoff between a quick decision and a reliable decision. The two alternative choice problem is a simplification of many decision-making scenarios and captures the essence of the speed-accuracy tradeoff in a variety of situations encountered by animal groups~\cite{JWWA-JEH-DJTS-JS:11,VG-IDC:10}. Moreover, human performance in two alternative choice tasks is extensively studied and well understood~\cite{RB-EB-etal:06,RR:78,RR-GM:08}. In particular, human performance in a two alternative choice task is modeled well by the \emph{drift-diffusion model} (DDM) and its variants;  variants of the DDM under optimal choice of parameters are equivalent to the DDM.

Collective decision-making in human groups is typically studied under two extreme communication regimes: the so-called \emph{ideal group} and the \emph{Condorcet group}. In an ideal group, each decision-maker interacts with every other decision-maker and the group arrives at a consensus decision.  In a Condorcet group, decision-makers do not interact with one another; instead a majority rule  is employed to reach a decision. { In this paper we study a generalization of the ideal group, namely, \emph{the ideal network}. In an ideal network, each decision-maker interacts only with a subset of other decision-makers in the group, and the group arrives at a consensus decision. }

Human decision-making is typically studied in two paradigms, namely, \emph{interrogation}, and \emph{free response}. In the interrogation paradigm, the human has to make a decision at the end of a prescribed time duration, while in the free response paradigm, the human takes his/her time to make a decision. { In the model for human decision-making in the interrogation (free response) paradigm, the decision-maker compares the decision-making evidence against a single threshold (two thresholds) and makes a decision. In the context of the free response paradigm, the choice of the thresholds dictates the speed-accuracy tradeoff in decision-making.}

Collective decision-making in ideal human groups and Condorcet human groups is studied in~\cite{RDS-CJH-RW:01} using the classical signal detection model for human performance in two alternative choice tasks. 
Collective decision-making in Condorcet human groups using the DDM and the free response paradigm is studied in~\cite{MK-JM:12,SHD-RC-FB:09y}.
Collective decision-making in ideal human groups using the DDM and the interrogation paradigm is studied in~\cite{IP-LS-NEL:12}. Related collective decision-making models in animal groups are studied in~\cite{DP-NEL:12}. In this paper, we study the free response paradigm for collective decision-making  in ideal networks using the DDM, the Ornstein-Uhlenbeck (O-U) model, and the race model~\cite{TM-PH:06}.  

The DDM is a continuum approximation to the evidence aggregation process in a hypothesis testing problem. Moreover, the hypothesis testing problem with fixed sample size and the sequential hypothesis testing problem correspond to the interrogation paradigm and the free response paradigm in human decision-making, respectively. 
Similarly, the race model in the free response paradigm is a continuum approximation of an asymptotically optimal sequential multiple hypothesis test proposed in~\cite{VPD-AGT-VVV:99}.
Consequently, the
collective decision-making problem in human groups is similar to distributed hypothesis testing problems studied in the engineering literature~\cite{RSB-SAK-HVP:97,PB-SM-VM-PW:10,ROS-EF-EF-JSS:06,DB-DJ-etal:11}. 
In particular, 
Braca~\etal~\cite{PB-SM-VM-PW:10} study distributed implementations of the hypothesis testing problem with fixed sample size as well as the sequential hypothesis testing problem. 
They use the \emph{running consensus} algorithm~\cite{PB-SM-VM:08} to aggregate the test statistic across the network and show that the proposed algorithm achieves the performance of a centralized algorithm asymptotically. 
{We rely on the Laplacian flow~\cite{ ROS-JAF-RMM:07} to aggregate evidence across the network. 
Our information aggregation model is the continuous time equivalent of the running consensus algorithm with fixed network structure. In contrast to showing the asymptotic behavior of the running consensus as in~\cite{PB-SM-VM-PW:10}, we characterize the finite time behavior of our information aggregation model  in the free response paradigm.}


An additional relevant literature concerns the sensor selection problem in decentralized decision-making. In the context of  sequential multiple hypothesis testing with a fusion center, a sensor selection problem to minimize the expected decision time for a prescribed error rate has been studied in~\cite{VS-KP-FB:08p}. 
In the present paper, 
for a prescribed error rate, we characterize the expected decision time as a function of the node centrality.
Such a characterization can be used to select a representative node that has the minimum expected decision time, among all nodes, for a prescribed error rate.

In this paper, we begin by studying the speed-accuracy tradeoff in collective decision-making using the context of two alternative choice tasks. We model the evidence aggregation across the network using the Laplacian flow based coupled DDM~\cite{IP-LS-NEL:12}. 
{ In order to determine the decision-making performance of each individual in the network,}
 we develop a decoupled approximation to the coupled DDM, namely, \emph{the reduced DDM}, and characterize its properties. We extend the coupled DDM to the context of decision-making in two alternative choice tasks with recency effects, i.e., tasks 
in which the recently collected evidence weighs more that the evidence collected earlier. We also extend the coupled DDM to the context of decision-making in
 multiple alternative choice tasks. 

The major contributions of this paper are fivefold. First, we propose a set of reduced DDMs, a novel decoupled approximation to the coupled DDM and characterize the efficiency of the approximation. Each reduced DDM comprises two components: (i) a centralized component common to each reduced DDM, and (ii) a local component that depends on the location of the decision-maker in the network. 

Second, we present partial differential equations (PDEs) to determine the expected decision time and the error rate for each reduced DDM. We also derive high probability bounds on the expected decision time and the error rate for the reduced DDMs. Our bounds rely on the first passage time properties of the O-U process. 

Third, we numerically demonstrate that, for large thresholds, the error rates and the expected decision times for the coupled DDM are approximated well by the corresponding quantities for a centralized DDM with modified thresholds. We also obtain an expression for threshold modifications (referred to as threshold corrections) from our numerical data. 

Fourth, we examine various threshold selection criteria and analyze the decision-making performance as a function of  the decision-maker's location in the network. Such an analysis is helpful in selecting representative nodes for high performance in decision-making, e.g., selecting a node that has the minimum expected decision time, among all nodes, for a prescribed maximum probability of error.  

Fifth, we extend the coupled DDM to the coupled O-U model and the coupled race model for collective decision-making in two alternative choice tasks with recency effects and multiple alternative choice tasks, respectively. 

The remainder of the paper is organized as follows. We review decision-making models for individual humans and human groups in Section~\ref{sec:models}. 
We present properties of the coupled DDM in Section~\ref{sec:prop-coupled-ddm}. We propose the { reduced} DDM and characterize its performance in the free response decision-making paradigm in Section~\ref{sec:reduced-ddm-free-response}. We present some numerical illustrations and results in Section~\ref{sec:numerical}. We examine various threshold selection criteria and the effect of the decision-makers's location in the network on their decision-making performance in Section~\ref{sec:threshold-design}. We extend the coupled DDM to the coupled O-U model and the coupled race model in Section~\ref{sec:other-models}.
Our conclusions are presented in Section~\ref{sec:conclusions}.

\section{Human Decision-Making Models}\label{sec:models}
In this section, we survey models for human decision-making. We present the drift diffusion model (DDM) and the coupled DDM  that capture  individual and network decision-making in a two alternative choice task, respectively.

\subsection{Drift Diffusion Model}\label{sec:sprt}
A two alternative choice task~\cite{RB-EB-etal:06}  is a decision-making scenario in which a decision-maker has to {choose} between two plausible alternatives. In a two alternative choice task, the difference between the log-likelihood of each alternative (evidence) is aggregated
and the aggregated evidence is compared against thresholds to make a decision. 
The evidence aggregation  is modeled well by the drift-diffusion process~\cite{RB-EB-etal:06} defined by
\begin{equation}\label{eq:ddm}
\mr d x (t) = \beta \mathrm{d}t + \sigma \mathrm{d}W(t), \quad x(0)=x_0,
\end{equation}
where $\beta \in \real$ and $\sigma  \in \real_{>0}$ are, respectively, the drift rate and the diffusion rate, $W(t)$ is the standard one-dimensional Weiner process, $x(t)$ is the aggregate evidence at time $t$, and $x_0$ is the initial evidence (see~\cite{RB-EB-etal:06} for the details of the model). 
{The two decision hypotheses correspond to the drift rate being positive or negative, i.e., $\beta \in \real_{>0}$ or $\beta \in \real_{<0}$, respectively.}

Human decision-making  is studied  in two paradigms, namely, \emph{interrogation} and \emph{free-response}. In the interrogation paradigm, a time duration is prescribed to the human who decides on an alternative at the end of this duration. In the model for the interrogation paradigm, by the end of the prescribed duration, the human compares the aggregated evidence against a single threshold, and chooses an alternative.
In the free response paradigm, the human subject is free to take as much time as needed to make a reliable decision. In the model for this paradigm, 
at  each time $\tau \in \real_{\ge 0}$, the human compares the aggregated evidence against two symmetrically chosen thresholds $\pm \eta, \eta \in \real_{\ge 0}$: (i) if $x(\tau) \ge  \eta$, then the human decides in favor of the first alternative; (ii) if  $x(\tau) \le -\eta$, then the human decides in favor of the second alternative; (iii) otherwise, the human collects more evidence. The DDM in the free response paradigm  is the continuum limit of the \emph{sequential probability ratio test}~\cite{AW:45} that requires a minimum expected number of observations to decide within a prescribed probability of error. 

{
A decision is said to be \emph{erroneous}, if it is in favor of an incorrect alternative. 
For the DDM and the free response decision-making paradigm, the \emph{error rate} is the probability of making an erroneous decision, and the \emph{decision time} is the time required to decide on an alternative.}  In particular, for $\beta \in \real_{> 0}$, the decision time $T$ is defined by
\[
T =\inf\setdef{t\in \real_{\ge 0}}{x(t) \in \{-\eta, +\eta\}},
\]
and the error rate $\mr{ER}$ is defined by $\mr{ER}=\prob(x(T)=-\eta)$. For the DDM~\eqref{eq:ddm} with thresholds $\pm \eta$, the expected decision time $\mr{ET}$ and the error rate $\mr{ER}$ are, respectively, given by~\cite{RB-EB-etal:06}
\begin{equation}\label{eq:ddm-et-er}
\mr{ET}=\frac{\eta}{\beta}\tanh\Big(\frac{\beta \eta}{\sigma^2}\Big)\quad \text{and}\quad \mr{ER} = \frac{1}{1+\exp\big(\frac{2\beta \eta}{\sigma^2}\big)}.
\end{equation}


\subsection{Coupled drift diffusion model}
Consider a set of $n$ decision-makers performing a two alternative choice task and let their interaction topology be modeled by a connected undirected graph $\mc G$ with Laplacian matrix $L \in \real^{n\times n}$. The evidence aggregation in collective decision-making is modeled in the following way.
 At each time $t\in\real_{\ge 0}$, every decision-maker $k\in \until{n}$ (i) computes a convex combination of her evidence with her neighbor's evidence; (ii) collects new evidence; and (iii) adds the new evidence to the convex combination. This collective evidence aggregation process is mathematically described by the following coupled drift diffusion model~\cite{IP-LS-NEL:12}:
\begin{equation}\label{coupled-ddm}
\mr d \bs x(t)= \big( \beta \bs 1_n  -L \bs x(t) \big) \mathrm{d}t + \sigma \mc I_n \mathrm{d}\bs W_n(t),
\end{equation}
where $\bs x(t) \in \real^n$ is the vector of aggregate evidence at time $t$, 
$\bs W_n(t) \in \real^n$ is the standard $n$-dimensional vector of Weiner processes, 
$\bs 1_n$ is the column $n$-vector of all ones, and $\mc I_n$ is the identity matrix of order $n$.  
{The two decision hypotheses correspond to the drift rate being positive or negative, i.e., $\beta \in \real_{>0}$ or $\beta \in \real_{<0}$, respectively.}

The coupled DDM~\eqref{coupled-ddm} captures the interaction among individuals using the Laplacian flow dynamics.  { The Laplacian flow is the continuous time equivalent of the classical DeGroot model~\cite{MOJ:10,MHDG:74}, which is a popular model for learning in social networks~\cite{BG-MOJ:10}. 
However, the social network literature employs the DeGroot model to reach a consensus on the belief of each individual~\cite{AJ-AS-ATS:10}, while the coupled DDM employs the Laplacian flow to achieve a consensus on the evidence available to each individual. 
The coupled DDM~\eqref{coupled-ddm} is the continuous time equivalent of the running consensus algorithm~\cite{PB-SM-VM:08} with a fixed interaction topology.}


The solution to the stochastic differential equation (SDE)~\eqref{coupled-ddm} is a Gaussian process, and for $\bs x(0) =\bs 0_n$, where $\bs 0_n$ is the column $n$-vector of all zeros, 
\begin{equation} \label{eq:evidence}
\begin{split}
\expt[x(t)] &= \beta t \bs 1_n, \\
\text{Cov}(x_k(t), x_j(t)) &= \frac{\sigma^2 t}{n} +\sigma^2 \sum_{p=2}^n \! \frac{1-e^{-2 \lambda_p t}}{2 \lambda_p}u_k^{(p)} u_j^{(p)},
\end{split}
\end{equation}
for $k,j \in \until{n}$, where $\lambda_p, \; p\in \{2,\ldots,n\}$,  are non-zero eigenvalues of the Laplacian matrix, and $u_k^{(p)}$ is the $k$-th component of the normalized eigenvector associated with eigenvalue $\lambda_p$ (see~\cite{IP-LS-NEL:12} for details). 


\begin{assumption}[\bit{Unity diffusion rate}]
In the following, without loss of generality, we assume the diffusion rate $\sigma =1$ in the coupled DDM~\eqref{coupled-ddm}.
Note that if the diffusion rate is non-unity, then it can be made unity by scaling the drift rate, the initial condition, and thresholds by $1/\sigma$.
 \oprocend
\end{assumption}

\begin{remark}[\bit{Ideal network as generalized ideal group}]
 In contrast to the standard ideal group analysis~\cite{RDS-CJH-RW:01} that assumes each individual interacts with every other individual, in~\eqref{coupled-ddm} each individual interacts only with its neighbors in the interaction graph $\mc G$. Thus, the coupled DDM~\eqref{coupled-ddm} generalizes the ideal group model and captures more general interactions, e.g., organizational hierarchies.  We refer to this decision-making system as an ideal network.  \oprocend
\end{remark}

\section{Properties of the Coupled DDM}\label{sec:prop-coupled-ddm}
In this section, we study properties of the coupled DDM. 
We first present the principal component analysis of the coupled DDM.
We then show that the coupled DDM is an asymptotically optimal decision-making model. 
We utilize the principal component analysis to decompose the coupled DDM into a centralized DDM and a coupled error dynamics. We then develop decoupled approximations of the error dynamics. 

\subsection{Principal component analysis of the coupled DDM}
In this section, we study the principal components of the coupled DDM. It follows from~\eqref{eq:evidence} that, for the coupled DDM~\eqref{coupled-ddm}, the covariance matrix of the evidence at time $t$ is
\begin{equation}\label{eq:covariance}
\text{Cov}(\bs x(t)) = \frac{t}{n} \bs 1_n \bs 1_n^{\top} + \sum_{p=2}^n \frac{1-e^{-2\lambda_p t}}{2 \lambda_p} \bs u_p \bs u_p^{\top},
\end{equation}
where $\bs u_p \in \real^n$ is the eigenvector of the Laplacian matrix $L$ corresponding to eigenvalue $\lambda_p \in \real_{>0}$.  Since
\[
 \frac{1-e^{-2\lambda_p t}}{2 \lambda_p} < t, \text{ for each } p \in \{2,\ldots,n\},
\]
it follows that the first principal component of the coupled DDM corresponds to the eigenvector $\bs 1_n /\sqrt{n}$, i.e., the first principle component  is a set of
%
identical DDMs, each of which is the average of the individual DDMs.
Such an averaged DDM, referred to as \emph{the centralized DDM},  is described as follows: 
\begin{equation}\label{eq:centralized-ddm}
\mr d \subscr{x}{cen}(t) = \beta \mr d t + \frac{1}{n}\bs 1_n^{\top} \mr d \bs W(t), \; \subscr{x}{cen}(0)=0.
\end{equation}
Other principal components correspond to the remaining component $\bs x(t) -\subscr{x}{cen}(t) \bs 1_n$ of the evidence that we define as the error vector $\bs \epsilon(t) \in \real^{n}$. It follows immediately that the error dynamics are
\begin{equation}\label{eq:error-dynamics}
\! \mr d \bs \epsilon (t) = - L \bs \epsilon(t) \mr dt +  (I_n - \frac{1}{n} \bs 1_n \bs 1_n^{\top}) \mr d \bs W_n(t),\bs \epsilon(0) =\bs 0_n.
\end{equation}

We summarize the above discussion in the following proposition.
\begin{proposition}[\bit{Principal components of the coupled DDM}]\label{prop:principal-components}
The coupled DDM~\eqref{coupled-ddm} can be decomposed into the centralized DDM~\eqref{eq:centralized-ddm} at each node and the error dynamics~\eqref{eq:error-dynamics}. Moreover, the centralized DDM is the first principal component of the coupled DDM and the error dynamics correspond to the remaining principal components. 
\end{proposition}

\subsection{Asymptotic optimality of the coupled DDM}
The centralized DDM~\eqref{eq:centralized-ddm} is the DDM in which all evidence is available at a fusion center. 
It follows from the optimality of the DDM in the free response paradigm that the centralized DDM is also optimal in the free response paradigm. We will show that the coupled DDM is asymptotically equivalent to the centralized DDM and thus asymptotically optimal.

\begin{proposition}[\bit{Asymptotic optimality}] \label{prop:asymptotic-optimality}
The evidence $x_k(t)$ aggregated at each node $k\in\until{n}$ in the coupled DDM  is equivalent to the  evidence $\subscr{x}{cen}(t)$ aggregated  in the centralized DDM  as $t\to +\infty$. 
\end{proposition}
\begin{proof}
We start by solving~\eqref{eq:error-dynamics}. Since~\eqref{eq:error-dynamics} is a linear SDE, the solution is a Gaussian process with $\expt[\bs \epsilon(t)] = 0$ and
\begin{align*}
\text{Cov}(\bs \epsilon(t))& = \int_{0}^t e^{-L(t-s)} (I_n - \frac{1}{n} \bs 1_n \bs 1_n^{\top}) e^{-L(t-s)} \mr ds\\
&=  \int_{0}^t  e^{-2 L s} \mr ds - \frac{ t}{n}\bs 1_n \bs 1_n^{\top}.
\end{align*}
Therefore, 
$ \displaystyle
\text{Var}(\epsilon_k(t)) =  \sum_{p=2}^n \frac{1-e^{-2\lambda_p t}}{2\lambda _p} {u_k^{(p)}}^2.
$

We further note that
\begin{align*}
\frac{x_k(t) -\beta t}{ \sqrt{t} } &=  \frac{\subscr{x}{cen}(t) -\beta t}{  \sqrt{t}} + \frac{\epsilon_k(t)}{ \sqrt{t}}\\
& = \frac{1}{\sqrt{n t} } W(t) + \frac{\sqrt{\sum_{p=2}^n \frac{1-e^{-2\lambda_p t}{u_k^{(p)}}^2}{2\lambda _p}}}{t} W(t) .
\end{align*}
Note that $\epsilon_k(t)$ can be written as a scaled Weiner process because it is an almost surely continuous Gaussian process. Therefore,{
\[
\frac{x_k(t) -\beta t}{ \sqrt{t} }  \to  \frac{1}{\sqrt{n t}} W(t)  = \frac{\subscr{x}{cen}(t)-\beta t}{\sqrt{t}},
\]
in distribution as $t\to+\infty$,} and the asymptotic equivalence follows. 
\end{proof}

{
\begin{remark}[\bit{Effectiveness of collective decision-making}]
In view of Proposition~\ref{prop:asymptotic-optimality}, for large thresholds, each node in the coupled DDM behaves like the centralized DDM. In the limit of large thresholds, it follows for~(\ref{eq:centralized-ddm}) from~\eqref{eq:ddm-et-er} that the expected decision time for each individual is approximately $\frac{\eta}{\beta}$, and the error rate is $\exp(-\frac{2\beta n \eta}{\sigma^2})$. Therefore, for a given large threshold, the expected decision-time is the same under collective decision-making and individual decision-making. However, the error rate decreases exponentially with increasing group size. \oprocend
\end{remark}
}

\begin{definition}[\bit{Node certainty index}]\label{def:node-certainty-index}
For the coupled DDM~\eqref{coupled-ddm} and node $k$, the node certainty index~\cite{IP-LS-NEL:12}, denoted by $\mu_k$, is defined as the inverse of the steady state error variance in~\eqref{eq:error-dynamics}, i.e.,
\[
\frac{1}{\mu_k} = \sum_{p=2}^n \frac{1}{2\lambda _p} {u_k^{(p)}}^2.
\]
\end{definition}

It has been shown in~\cite{IP-LS-NEL:12} that the node certainty index is equivalent to the \emph{information centrality}~\cite{KS-MZ:89} which is defined as the inverse of the harmonic mean of the \emph{effective} path lengths from the given node to every other node in the interaction graph. Furthermore, it can be verified that $\mu_k \ge {2n \lambda_{2\text{-min}}}/{(n-1)}$, where $\lambda_{2\text{-min}}$ is the smallest positive eigenvalue of the Laplacian matrix associated with the interaction graph. 

\subsection{Decoupled approximation to the error dynamics}
We examine the free response paradigm for the coupled DDM, which corresponds to 
the boundary crossing of the $n$-dimensional Weiner process with respect to the thresholds $\pm \eta$. In general, for $n>1$, boundary crossing properties of the Weiner process are hard to characterize analytically. Therefore, we resort to approximations for the coupled DDM.  In particular, we are interested in mean-field type approximations~\cite{DK-NF:09}
that reduce a coupled system with $n$ components to a system with $n$ decoupled components. 

We note that  the error at node $k$ is a Gaussian process with zero mean and steady state variance $1/\mu_k$. In order to approximate the coupled DDM with $n$ decoupled systems, we approximate the error dynamics at node $k$ by the following Ornstein-Uhlenbeck (O-U) process
\begin{equation}\label{eq:error-o-u}
\mr d \varepsilon_k(t) = -\frac{\mu_k}{2} \varepsilon_k(t) + \mr d W(t), \quad \varepsilon_k(0)=0,
\end{equation}
for each $k\in \until{n}$. Note that different nodes will have different realizations of the 
Weiner process $W(t)$ in~\eqref{eq:error-o-u}; however, for simplicity of notation, we do not use the index $k$ in $W(t)$. 
We now study the efficiency of such an approximation. We first introduce some notation. Let $\bs \mu \in \real_{>0}^n$ be the vector of $\mu_k, k\in \until{n}$. Let $\text{diag}(\cdot)$ represent the diagonal matrix with its argument as the diagonal entries. Let $\tilde \lambda_p$ be the $p$-th eigenvalue of $L+\text{diag}(\bs \mu/2)$ and let ${\tilde {\bs u}}^{(p)}$ be the associated eigenvector.

\begin{proposition}[\bit{Efficiency of the error approximation}]\label{prop:efficiency}
For the coupled error dynamics~\eqref{eq:error-dynamics} and the decoupled approximate error dynamics~\eqref{eq:error-o-u}, the following statements hold:
\begin{enumerate}
\item the expected error $\expt[\epsilon_k(t)]$ and  $\expt[\varepsilon_k(t)]$ are zero uniformly in time;
\item the error variances $\expt[\epsilon_k(t)^2]$ and $\expt[\varepsilon_k(t)^2]$   converge exponentially to $\frac{1}{\mu_k}$;
\item the steady state correlation between $\epsilon_k(t)$ and $\varepsilon_k(t)$ is
\[
\lim_{t\to +\infty} \text{corr}(\epsilon_k(t), \varepsilon_k(t)) = \mu_k \sum_{p=1}^n \frac{1}{2\tilde \lambda_p} ({\tilde u}_k^{(p)})^2 -\frac{2}{n}.
\]
\end{enumerate}
\end{proposition}
\begin{proof}
The combined error and approximate error dynamics are
\begin{equation}\label{eq:combined-error-dynamics}
\mr d \begin{bmatrix} \bs \epsilon(t) \\ \bs {\varepsilon}(t) \end{bmatrix}
= -\begin{bmatrix} L \bs \epsilon(t) \\ \text{diag}(\frac{\bs \mu}{2}) \bs {\varepsilon}(t) \end{bmatrix} \mr dt
+ \begin{bmatrix} I_n -\frac{1}{n} \bs 1_n \bs 1_n^{\top} \\ I_n \end{bmatrix} \mr d \bs W_n(t),
\end{equation}
where $\bs {\varepsilon}(t)$ is the vector of $\varepsilon_k(t), k\in \until{n}$.

The combined error dynamics~\eqref{eq:combined-error-dynamics} is a linear SDE that can be solved in closed form to obtain
\[
\expt\begin{bmatrix} \bs \epsilon(t) \\ \bs {\varepsilon}(t) \end{bmatrix} = \exp\big(-\big[\begin{smallmatrix} L & 0\\ 0& \text{diag}(\frac{\bs \mu}{2})\end{smallmatrix}\big] t\big)  \begin{bmatrix} \bs \epsilon(0) \\ \bs {\varepsilon}(0) \end{bmatrix} = \begin{bmatrix} \bs 0_n \\ \bs 0_n \end{bmatrix}.
\]
This establishes the first statement. 

We further note that the covariance of $\big[\begin{smallmatrix} \bs \epsilon(t) \\ \bs {\varepsilon}(t) \end{smallmatrix}\big]$ is 
\begin{multline}\label{eq:combined-error-covariance}
\text{Cov} \big(\big[\begin{smallmatrix} \bs \epsilon(t) \\ \bs {\varepsilon}(t) \end{smallmatrix}\big]\big)
= \int_{0}^t  \exp\big(-\big[\begin{smallmatrix} L & 0\\ 0 & \text{diag}(\frac{\bs \mu}{2})\end{smallmatrix}\big] s \big)
\big[\begin{smallmatrix} I_n -\frac{1}{n} \bs 1_n \bs 1_n^{\top} \\ I_n \end{smallmatrix}\big]\\
\big[\begin{smallmatrix} I_n -\frac{1}{n} \bs 1_n \bs 1_n^{\top} & I_n \end{smallmatrix}\big] 
\exp\big(-\big[\begin{smallmatrix} L & 0 \\0 & \text{diag}(\frac{\bs \mu}{2})\end{smallmatrix}\big] s \big) \mr d s.
\end{multline}
Some algebraic manipulations reduce~\eqref{eq:combined-error-covariance} to
\begin{align*}
\text{Cov}(\bs \epsilon(t))& = \int_{0}^t e^{-2 L s} \mr d s - \frac{t}{n} \bs 1_n \bs 1_n^{\top},\\
\text{Cov}(\varepsilon_k(t), \varepsilon_k(t)) &= \frac{1-e^{-\mu_k t}}{\mu_k}, \quad \text{and}\\
\expt[\bs \epsilon(t) \bs {\varepsilon}(t)^{\top}] &= \int_{0}^t  \big( e^{-(L + \frac{\text{diag}(\bs \mu)}{2})s}  -\frac{1}{n} \bs 1 \bs 1_n^{\top} e^{-\frac{\text{diag}(\bs \mu)}{2} s} \big) \mr ds.
\end{align*}
It immediately follows that the steady state variance of $\epsilon_k$ and $\varepsilon_k$ is $1/\mu_k$. This establishes the second statement. 

To establish the third statement, we simplify the expression for $\expt[\bs \epsilon(t) \bs {\varepsilon}(t)^{\top}] $ to obtain
\[
\expt[\epsilon_k(t) \varepsilon_k(t)] =\sum_{p=1}^n \frac{1-e^{-2\tilde \lambda_p t}}{2\tilde \lambda_p} ({{\tilde u}_k}^{(p)})^2 - \frac{2(1- e^{-\frac{\mu_k t}{2}})}{n \mu_k}.
\]
Thus, the steady state correlation is
\begin{align*}
 \lim_{t\to+\infty} 
\frac{\expt[\epsilon_k(t) \varepsilon_k(t)] }{\sqrt{\expt[\epsilon_k(t)^2] \expt[{\varepsilon_k(t)}^2] }}
&= \mu_k \sum_{p=1}^n \frac{1}{2\tilde \lambda_p} ({\tilde u}_k^{(p)})^2 -\frac{2}{n},
\end{align*}
and this establishes the proposition.
\end{proof}

\begin{remark}[\bit{Efficiency of the error approximation}]
For a large well connected network, the matrix $L+\frac{\text{diag}(\bs \mu)}{2}$ will be dominated by $\frac{\text{diag}(\bs \mu)}{2}$ and accordingly its eigenvectors will be close to the standard basis in $\real^n$. Thus, the steady state correlation between $\epsilon_k$ and $\varepsilon_k$ will be approximately $1-\frac{2}{n} \approx 1$, and the error approximation will be fairly accurate. \oprocend
\end{remark}

\section{Reduced DDM: Free Response Paradigm} \label{sec:reduced-ddm-free-response}

In this section we use the O-U process based error approximation~\eqref{eq:error-o-u} to develop an approximate information aggregation model for each node that we henceforth refer to as the \emph{reduced DDM}. We then present partial differential equations for the decision time and the error rate for the reduced DDM. Finally, we derive high probability bounds on the error rate and the expected decision time for the reduced DDM.  We study the free response paradigm under the following assumption:
\begin{assumption}[\bit{Persistent Evidence Aggregation}]
Each decision-maker continues to aggregate and communicate evidence according to the coupled DDM~\eqref{coupled-ddm} even after reaching a decision. \oprocend
\end{assumption}

\subsection{The reduced DDM}
We utilize the approximate error dynamics~\eqref{eq:error-o-u} to define  the reduced DDM at each node. 
The reduced DDM at node $k$ at time $t$ computes the evidence $y_k(t)$ by adding the approximate error $\varepsilon_k(t)$ to the evidence $\subscr{x}{cen}(t)$ aggregated by the centralized DDM. Accordingly, the reduced DDM at node $k$ is modeled by the following SDE:
\begin{equation}\label{eq:reduced-ddm}
\begin{bmatrix} \mr d  y_k(t) \\ \mr d  \varepsilon_k(t) \end{bmatrix}
= \begin{bmatrix}\beta  -\frac{\mu_k \varepsilon_k(t)}{2}  \\ -\frac{\mu_k  \varepsilon_k(t) }{2} \end{bmatrix} \mr d t 
+  \begin{bmatrix} \frac{1}{\sqrt{n}} & 1 \\ 0 & 1 \end{bmatrix} 
\begin{bmatrix} \mr d  W_1(t) \\ \mr d  W_2(t) \end{bmatrix},
\end{equation}
where $W_1(t)$ and $W_2(t)$ are independent standard one-dimensional Weiner processes. 

In the free response paradigm, 
decision-maker $k$ makes a  decision whenever $y_k(t)$ crosses one of the thresholds $\pm \eta_k$ for the first time.   If $\beta \in \real_{>0}$ and  $y_k(t)$ crosses threshold $+\eta_k (-\eta_k)$, then the decision is made in favor of the correct (incorrect) alternative. { Note that even though each individual in the network is identical, the model allows for them to have different thresholds. For simplicity, we consider symmetric thresholds for each individual; however, the following analysis holds for asymmetric thresholds as well.}

\subsection{PDEs for the decision time and the error rate}

The error rate is the probability of deciding in favor of an incorrect alternative.   The decision time is the time required to decide on an alternative. If $\beta \in \real_{>0}$ ($\beta \in \real_{<0}$), then an erroneous decision is made if the evidence crosses the threshold $-\eta_k$ ($+\eta_k$) before crossing the threshold $+\eta_k$ ($-\eta_k$). 
Without loss of generality, we assume that $\beta \in \real_{>0}$. We denote the error rate and the decision time for the $k$-th individual by $\mr{ER}_k$ and $T_k$, respectively. We denote the expected decision time at node $k$ by $\mr{ET}_k$.

We now determine the error rate and the expected decision time for the free response paradigm associated with the reduced DDM~\eqref{eq:reduced-ddm}.  For an SDE of the form~\eqref{eq:reduced-ddm}, the error rate and the decision time are typically characterized by solving the first passage time problem for the associated Fokker-Planck equation~\cite{CG:09}. For a homogeneous Fokker-Planck equation, the mean first passage time and the probability to cross the incorrect boundary before crossing the correct boundary is characterized by PDEs with initial conditions as variables. We now recall such PDEs and for completeness, we also present a simple and direct  derivation of these PDEs that does not involve the Fokker-Plank equation.  

\begin{proposition}[\bit{PDEs for error rate and decision time}] \label{prop:er-dt-pdes}
For the reduced DDM~\eqref{eq:reduced-ddm} with arbitrary initial conditions $y_k(0)=y_k^0\in[-\eta_k, \eta_k]$ and $\varepsilon_k(0)=\varepsilon_k^0\in [-\bar \eta_k, \bar \eta_k],  \bar \eta_k \to +\infty$, the following statements hold:
\begin{enumerate}
\item the partial differential equation for the expected decision time is
\begin{multline*}
\Big(\beta  -\frac{\mu_k  \varepsilon_k^0}{2}\Big) \frac{\partial \mr{ET}_k}{\partial y_k^0}   -\frac{\mu_k  \varepsilon_k^0}{2} \frac{\partial \mr{ET}_k}{\partial  \varepsilon_k^0}\\
+ \frac{1}{2} \Big( \frac{n+1}{n} \frac{\partial^2 \mr{ET}_k}{\partial {y_k^0}^2}+ 2 \frac{\partial^2 \mr{ET}_k}{\partial {y_k^0} \partial \varepsilon^0_k}+  \frac{\partial^2 \mr{ET}_k}{\partial {\varepsilon_k^0}^2}\Big)=-1,
\end{multline*}
with boundary conditions $\mr{ET}_k(\cdot, \pm \eta_k) =0$, and $\mr{ET}_k(\pm \bar \eta_k, \cdot) =0$; 
\item the partial differential equation for the error rate is
\begin{multline*}
\Big(\beta  -\frac{\mu_k  \varepsilon_k^0}{2}\Big) \frac{\partial \mr{ER}_k}{\partial y_k^0}   -\frac{\mu_k  \varepsilon_k^0}{2} \frac{\partial \mr{ER}_k}{\partial  \varepsilon_k^0}\\
+ \frac{1}{2} \Big( \frac{n+1}{n} \frac{\partial^2 \mr{ER}_k}{\partial {y_k^0}^2}+ 2 \frac{\partial^2 \mr{ER}_k}{\partial {y_k^0} \partial \varepsilon^0_k}+  \frac{\partial^2 \mr{ER}_k}{\partial {\varepsilon_k^0}^2}\Big)=0,
\end{multline*}
with boundary conditions $\mr{ER}_k(\cdot, \eta_k)=0$, $\mr{ER}_k(\cdot, -\eta_k)=1$, $\mr{ER}_k(\bar \eta_k, \cdot)=0$, and $\mr{ER}_k(-\bar \eta_k, \cdot)=1$.
\end{enumerate}
\end{proposition}
\begin{proof}
We start with the first statement. Let the expected decision time for the reduced DDM with initial condition $(y_k^0, \varepsilon_k^0)$ be $\mr{ET}_k(y_k^0, \varepsilon_k^0)$. Consider the evolution of the reduced DDM over an infinitesimally small duration $h\in \real_{>0}$. Then
\begin{equation*}
\begin{bmatrix} \mr y_k(h)-y_k^0  \\ \mr  \varepsilon_k(h)- \varepsilon_k^0 \end{bmatrix}
= \begin{bmatrix}\beta  -\frac{\mu_k \varepsilon_k^0}{2}  \\ -\frac{\mu_k  \varepsilon_k^0}{2} \end{bmatrix} h
+  \begin{bmatrix} \frac{1}{\sqrt{n}} & 1 \\ 0 & 1 \end{bmatrix} \begin{bmatrix} W_1(h) \\ W_2(h) \end{bmatrix}.
\end{equation*}
By continuity of the trajectories of the reduced DDM it follows that $\mr{ET}_k(y_k^0, \varepsilon_k^0)= h+ \expt[ \mr{ET}_k(y_k(h), \varepsilon_k(h))]$, where  the expectation is over different realizations of 
$(y_k(h), \varepsilon_k(h))$. It follows from Taylor series expansion that
\begin{multline*}
\expt[ \mr{ET}_k(y_k(h), \varepsilon_k(h))] - \mr{ET}_k(y_k^0, \varepsilon_k^0) 
= \Big(\beta  -\frac{\mu_k  \varepsilon_k^0}{2}\Big) \frac{\partial \mr{ET}_k}{\partial y_k^0} h  \\ -\frac{\mu_k  \varepsilon_k^0}{2} \frac{\partial \mr{ET}_k}{\partial  \varepsilon_k^0}h 
 + \frac{1}{2} \Big( \frac{\partial^2 \mr{ET}_k}{\partial {y_k^0}^2} \expt[\frac{1}{n} W_1(h)^2 +W_2(h)^2]\\
 + 2 \frac{\partial^2 \mr{ET}_k}{\partial {y_k^0} \partial \varepsilon^0_k}\expt[W_2(h)^2]+  \frac{\partial^2 \mr{ET}_k}{\partial {\varepsilon_k^0}^2}W_2(h)^2\Big) +o(h^2),
\end{multline*}
where $o(h^2)$ represents terms of order $h^2$. Substituting $\mr{ET}_k(y_k^0, \varepsilon_k^0)= h+ \expt[ \mr{ET}_k(y_k(h), \varepsilon_k(h))]$, and $\expt[W_1(h)^2]=\expt[W_2(h)^2]=h$
 in the above expression, the PDE for the expected decision time follows. 
The boundary conditions follow from the definition of the decision time.
The PDE for the error rate follows similarly.
\end{proof}

The expected decision time and the error rate can be computed using the PDEs in Proposition~\ref{prop:er-dt-pdes}. These PDEs are nonlinear and to the best of our knowledge do not admit a closed form solution; however, they can be efficiently solved using numerical methods.

\subsection{Bounds on the expected decision time and the error rate}

In order to gain analytic insight into the behavior of the expected decision time and the error rate for the reduced DDM, we derive bounds on them that hold with high probability. We first recall the following result from~\cite{AGN-LMR-LS:85} on the first passage time density of the O-U process for large thresholds.
Let $\map{\subscr{T}{pass}^k}{\real}{\real_{\ge 0}}$ be the first passage time for the O-U process~\eqref{eq:error-o-u} as a function of the threshold. Moreover, let $\subscr{\bar T}{pass}^k$ denote the mean value of $\subscr{T}{pass}^k$.

\begin{lemma}[\bit{First passage time density for O-U process}]\label{lem-first-passage-density}
For the O-U process~\eqref{eq:error-o-u}, and a large threshold $\eta^{\varepsilon}_k \in \real_{>0}$, the first passage time density is
\[
f(\eta^{\varepsilon}_k, t | \varepsilon^0_k=0) = \frac{1}{\subscr{\bar{T}}{pass}^k (\eta^{\varepsilon}_k)} \exp\Big( -\frac{t}{\subscr{\bar{T}}{pass}^k(\eta^{\varepsilon}_k)}\Big),
\]
where  $\displaystyle 
\subscr{\bar{T}}{pass}^k(\eta^{\varepsilon}_k)= \frac{2}{\mu_k}\Big(\sqrt{\pi} \varphi\Big(\frac{\eta^{\varepsilon}_k\sqrt{\mu_k}}{\sqrt{2}} \Big) + \psi \Big(\frac{\eta^{\varepsilon}_k\sqrt{\mu_k}}{\sqrt{2}} \Big)\Big)$,
$\displaystyle
\varphi(z)=\int_{0}^z e^{\tau^2} \mr d\tau,\; \text{and } \psi(z) =\int_{0}^z \int_{0}^{\tau} e^{\tau^2} e^{-s^2} \mr ds \mr d \tau.
$
\end{lemma}

Lemma~\ref{lem-first-passage-density} suggests that for large thresholds the first passage time density is an exponential random variable. However, 
the expression for the mean first passage time does not provide much insight into the behavior of the first passage time density. In order to gain further insight into the first passage time density, we derive the following bounds on the mean first passage time.
\begin{lemma}[\bit{Bounds on the O-U mean first passage time}]\label{lem:bounds-first-passage-time}
For the O-U process~\eqref{eq:error-o-u}, and a large threshold $\eta^{\varepsilon}_k$, the mean first passage time $\subscr{\bar T}{pass}(\eta_k^{\varepsilon})$ satisfies:
\begin{align*}
\subscr{\bar T}{pass}^k(\eta_k^{\varepsilon}) & \le \frac{3\sqrt{\pi} \eta^{\varepsilon}_k}{\sqrt{2 \mu_k}} e^{{\eta^{\varepsilon}_k}^2 \mu_k/2}, \text{ and}\\
\subscr{\bar T}{pass}^k(\eta_k^{\varepsilon}) & \ge \frac{2}{\mu_k}\Big(
\frac{\sqrt{\pi}(e^{{\eta^{\varepsilon}_k}^2 \mu_k/2}-1)}{\sqrt{2}\eta^{\varepsilon}_k \sqrt{\mu_k}} 
+\frac{e^{{\eta^{\varepsilon}_k}^2 \mu_k/2}-1}{{2 \eta^{\varepsilon}_k}^2 \mu_k} -\frac{1}{2}\Big).
\end{align*}
\end{lemma}

\begin{proof}
We start by bounding function $\varphi$. We note that
\begin{align*}
\varphi(z) \ge \int_{0}^z \frac{\tau}{z} e^{\tau^2} \mr d \tau =\frac{(e^{z^2}-1)}{2z}.
\end{align*}
Moreover, it trivially follows that
$\varphi(z) \le z e^{z^2}.$

We now derive bounds on $\psi$. We note that
\begin{align*}
\psi(z)  &\ge \int_{0}^z \frac{e^{\tau^2}}{2\tau} \int_{0}^{\tau} 2s e^{-s^2} \mr ds \mr d \tau
= \int_{0}^z \frac{e^{\tau^2}-1}{2\tau} \mr d \tau \\
&\ge \frac{1}{2z}(\varphi(z)-z) \ge \frac{(e^{z^2}-1)}{4z^2} -\frac{1}{2}.
\end{align*}
Furthermore, using the bounds on the error function from equation~(7.1.14) of~\cite{MA-IAS:64}, we obtain
 \begin{align*}
\psi(z)  &\le \int_{0}^z  e^{\tau^2} \Big( \frac{\sqrt{\pi}}{2} -\frac{e^{-\tau^2}}{\tau+ \sqrt{\tau^2+2 }}\Big) \mr d \tau \le  \frac{\sqrt{\pi}}{2} z e^{z^2}.
\end{align*}
Substituting these lower and upper bounds in the expression for $\subscr{\bar T}{pass}^k (\eta_k^{\varepsilon})$ in Lemma~\ref{lem-first-passage-density}, we obtain the expressions in the lemma.
\end{proof}

We now derive high probability bounds on the O-U process~\eqref{eq:error-o-u}. 
Since the bounds on the first passage time are dominated by the term $e^{{\eta^{\varepsilon}_k}^2 \mu_k/2}$, in the following we explore bounds associated with thresholds of the form $\eta_k^{\varepsilon}=K/\sqrt{\mu_k}, K\in \real_{>0}$ so that the probability of crossing these thresholds does not vary much with the centrality of the node. 
Before we derive high probability bounds, we define function $\map{\subscr{\prob}{lower}}{\real_{> 0}\times\real_{> 0}\times\real_{\ge 0} }{(0,1)}$ by
\begin{multline*}
\subscr{\prob}{lower}(K,\mu_k, t)=\\
\exp\Big(-t \Big( \frac{\mu_k}{2}\Big(
\frac{\sqrt{\pi}(e^{\frac{K^2}{2}}-1)}{\sqrt{2} K} 
+\frac{e^{\frac{K^2}{2}}-1}{2 K^2} -\frac{1}{2}\Big)\Big)^{-1} \Big).
\end{multline*}
\begin{lemma}[\bit{Uniform bounds for the O-U process}]\label{lem:o-u-uniform-bounds}
For the error dynamics~\eqref{eq:error-o-u} and a large constant $K\in \real_{>0}$, the following statements hold:
\begin{enumerate}
\item the probability that the error $\varepsilon_k$ is uniformly upper bounded by $\frac{K}{\sqrt{\mu_k}}$ until time $t$ is
\begin{align*}
\prob\Big(\max_{s\in[0,t]} \varepsilon_k(s) &\le \frac{K}{\sqrt{\mu_k}} \Big) 
\ge  \subscr{\prob}{lower}(K,\mu_k,t);
\end{align*}
\item the probability that the error $\varepsilon_k$ is uniformly bounded in the interval $[-\frac{K}{\sqrt{\mu_k}},\frac{K}{\sqrt{\mu_k}}] $ until time $t$ is
\begin{align*}
\prob\Big(\max_{s\in[0,t]} |\varepsilon_k(s)| \le \frac{K}{\sqrt{\mu_k}} \Big) 
\ge 2 \subscr{\prob}{lower}(K,\mu_k,t) -1.
\end{align*}
\end{enumerate}
\end{lemma}
\begin{proof}
We start by with establishing the first statement. We note that
\begin{align*}
\prob\Big(\max_{s\in[0,t]} \varepsilon_k(s) \ge \frac{K}{\sqrt{\mu_k}} \Big) =\prob \Big(\subscr{T}{pass}^k \Big(\frac{K}{\sqrt{\mu_k}}\Big) \le t\Big).
\end{align*}
It follows from Lemma~\ref{lem-first-passage-density} that 
\[
\prob \Big(\subscr{T}{pass}^k \Big(\frac{K}{\sqrt{\mu_k}}\Big) \le t\Big) \le 1 - \exp\Big(\frac{-t}{\subscr{\bar T}{pass}^k (\frac{K}{\sqrt{\mu_k}})}\Big). 
\]
Substituting, the lower bound on the mean first passage time from Lemma~\ref{lem:bounds-first-passage-time}, we obtain
\[
\prob \Big(\subscr{T}{pass}^k \Big(\frac{K}{\sqrt{\mu_k}}\Big) \le t\Big) \le 1 - \subscr{\prob}{lower} (K, \mu_k,t),
\]
and the first statement follows.

By symmetry of the O-U process~\eqref{eq:error-o-u} about $\varepsilon_k=0$, it follows that
\[
\prob \Big(\subscr{T}{pass}^k\Big(-\frac{K}{\sqrt{\mu_k}}\Big) \le t\Big) \le 1 - \subscr{\prob}{lower} (K, \mu_k,t). 
\]
It follows from union bounds that
\[
\prob\Big(\max_{s\in[0,t]} |\varepsilon_k(s)| \ge \frac{K}{\sqrt{\mu_k}} \Big) \le 2(1 - \subscr{\prob}{lower} (K, \mu_k,t)), 
\]
and the second statement follows. 
\end{proof}

We now utilize the uniform bounds for the O-U process to derive bounds on the expected decision time and error rate for the reduced DDM. 
\begin{proposition}[\bit{Performance bounds for the reduced DDM}]\label{prop:bounds}
For the reduced DDM~\eqref{eq:reduced-ddm} at node $k$ with large thresholds $\pm\eta_k$ and a sufficiently large constant $K\in \real_{>0}$, the following statements hold with probability higher than $2\subscr{\prob}{lower} (K, \mu_k, \mr{ET}_k))-1$:
\begin{enumerate}
\item the expected decision time satisfies
\begin{multline*}
\frac{\eta_k - \frac{K}{\sqrt{\mu_k}}}{\beta}\tanh \Big(\beta n \Big(\eta_k - \frac{K}{\sqrt{\mu_k}}\Big)\Big) \le \mr{ET}_k \\
\le \frac{\eta_k + \frac{K}{\sqrt{\mu_k}}}{\beta}\tanh \Big(\beta n \Big(\eta_k + \frac{K}{\sqrt{\mu_k}}\Big)\Big);
\end{multline*}
\item the error rate satisfies
\begin{multline*}
\frac{1}{ 1+ \exp\big({2\beta n}\big(\eta_k + \frac{K}{\sqrt{\mu_k}}\big)\big) } \le \mr{ER}_k  \\
 \le 
\frac{1}{ 1+ \exp\big({2\beta n}\big(\eta_k - \frac{K}{\sqrt{\mu_k}}\big)\big) }.
\end{multline*}
\end{enumerate}
\end{proposition}
\begin{proof}
It follows from Lemma~\ref{lem:o-u-uniform-bounds} that until a given time $t$, the error process~\eqref{eq:error-o-u} belongs to the set $[-K/\sqrt{\mu}, K/\sqrt{\mu}]$ with probability greater than $2\subscr{\prob}{lower}(K, \mu_k, t)-1$. For the reduced DDM at node $k$, the time of interest is the decision time $t=T_k$. 
Furthermore, $\subscr{\prob}{lower}(K, \mu_k, t)$ is a convex function of $t$. Hence, from the Jensen inequality, the error is bounded in the set 
$[-K/\sqrt{\mu}, K/\sqrt{\mu}]$ at the time of decision with probability greater than $2\subscr{\prob}{lower} (K, \mu_k, \mr{ET}_k))-1$. This implies the effective threshold for the centralized DDM component in the reduced DDM at node $k$ is greater than $\eta_k -K/\sqrt{\mu_k}$ and smaller than $\eta_k + K/\sqrt{\mu_k}$ with probability greater than $2\subscr{\prob}{lower} (K, \mu_k, \mr{ET}_k))-1$. 
Since the decision time increases with increasing threshold and the error rate decreases with increasing threshold, inequalities for the decision time and the error rate follow from the corresponding expressions  in~\eqref{eq:ddm-et-er}. 
\end{proof}

\section{Numerical Illustrations}\label{sec:numerical}

Consider a set of nine decision-makers and let their interaction topology be modeled by the graph shown in Figure~\ref{fig:graph}. 
For this graph, the node certainty indices are $\mu_1=8.1$, $\mu_2=\mu_3=\mu_4=\mu_5=4.26$, and $\mu_6=\mu_7=\mu_8=\mu_9=1.6$.  
We first compare the performance of the reduced DDM with the coupled DDM. 
We pick the 
drift  rate $\beta$ at each node as $0.1$. We obtained error rates and decision times at node $6$ for the coupled DDM  and the reduced DDM through Monte-Carlo simulations, and we compare them  in Figure~\ref{fig:ddm-reduced}. 
Note that throughout this section, for better pictorial depiction, we plot the log-likelihood ratio of no error $\log (\frac{1-\mr{ER}_k}{\mr{ER}_k})$ instead of the error rate $\mr{ER}_k$. The log-likelihood ratio of no error decreases monotonically with the error rate. 
We also computed first passage time distributions at node $6$ for the coupled DDM and the reduced DDM with a threshold equal to $3$, and we compare them in Figure~\ref{fig:first-passage-comp}. 
It can be seen that the performance of the reduced DDM approximates the performance of  the coupled DDM very well. 

\begin{figure}[ht!]
\centering
\includegraphics[width=0.5\linewidth]{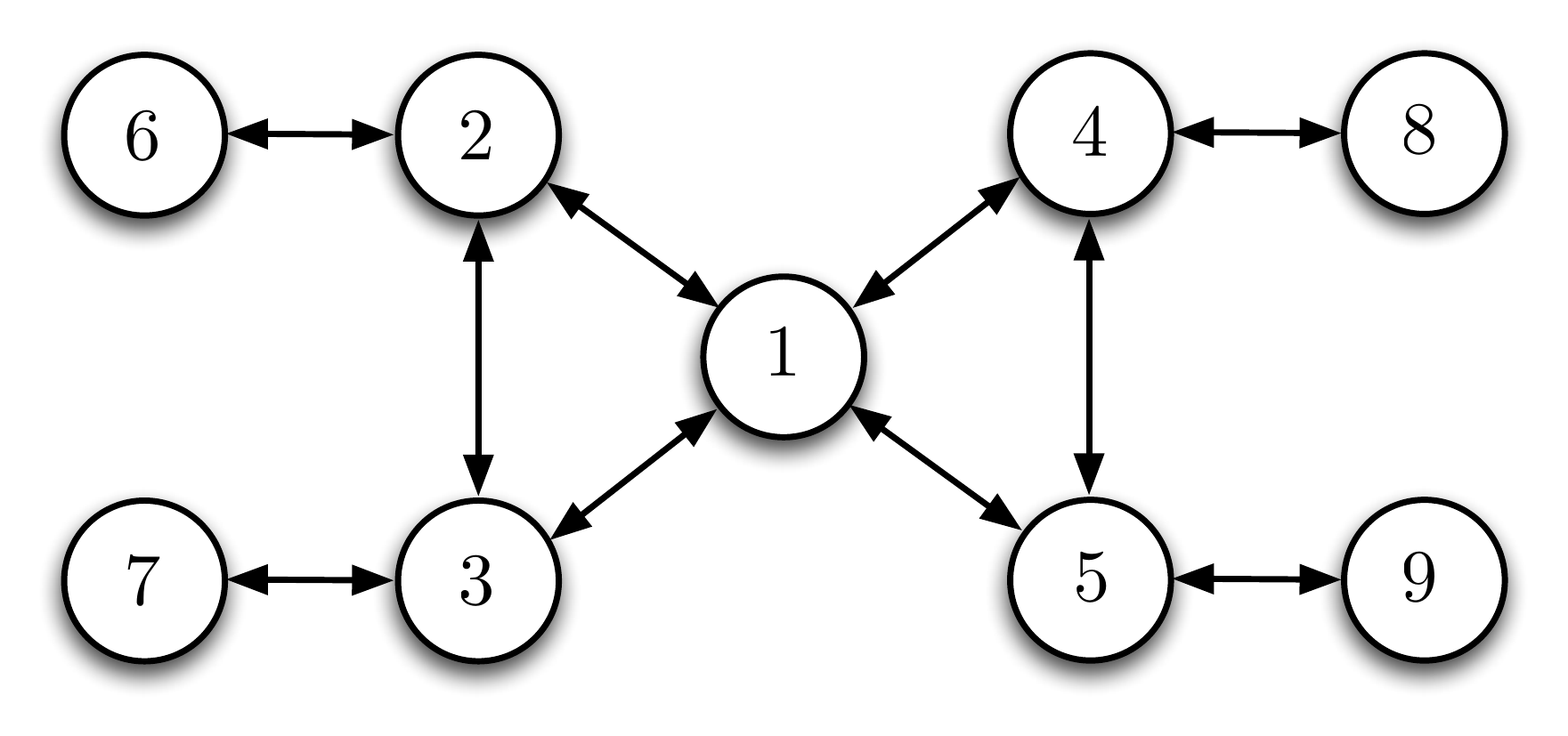}
\caption{Interaction graph for decision-makers. \label{fig:graph}}
\end{figure}

\begin{figure}
\centering 
\subfigure[Log-likelihood ratio of no error]{
\includegraphics[width=0.23\textwidth]{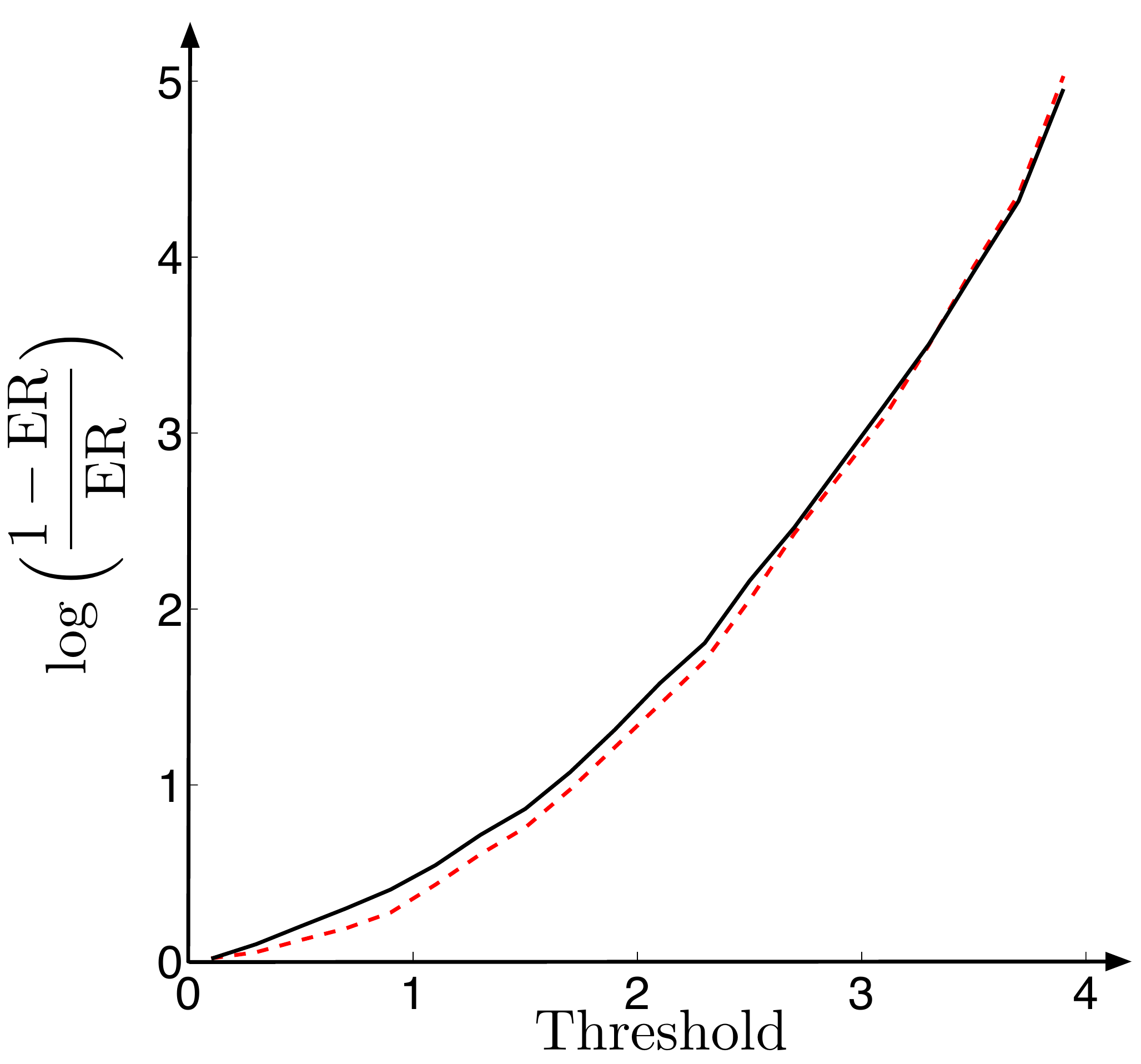}}
\subfigure[Expected decision times]{
\includegraphics[width=0.23\textwidth]{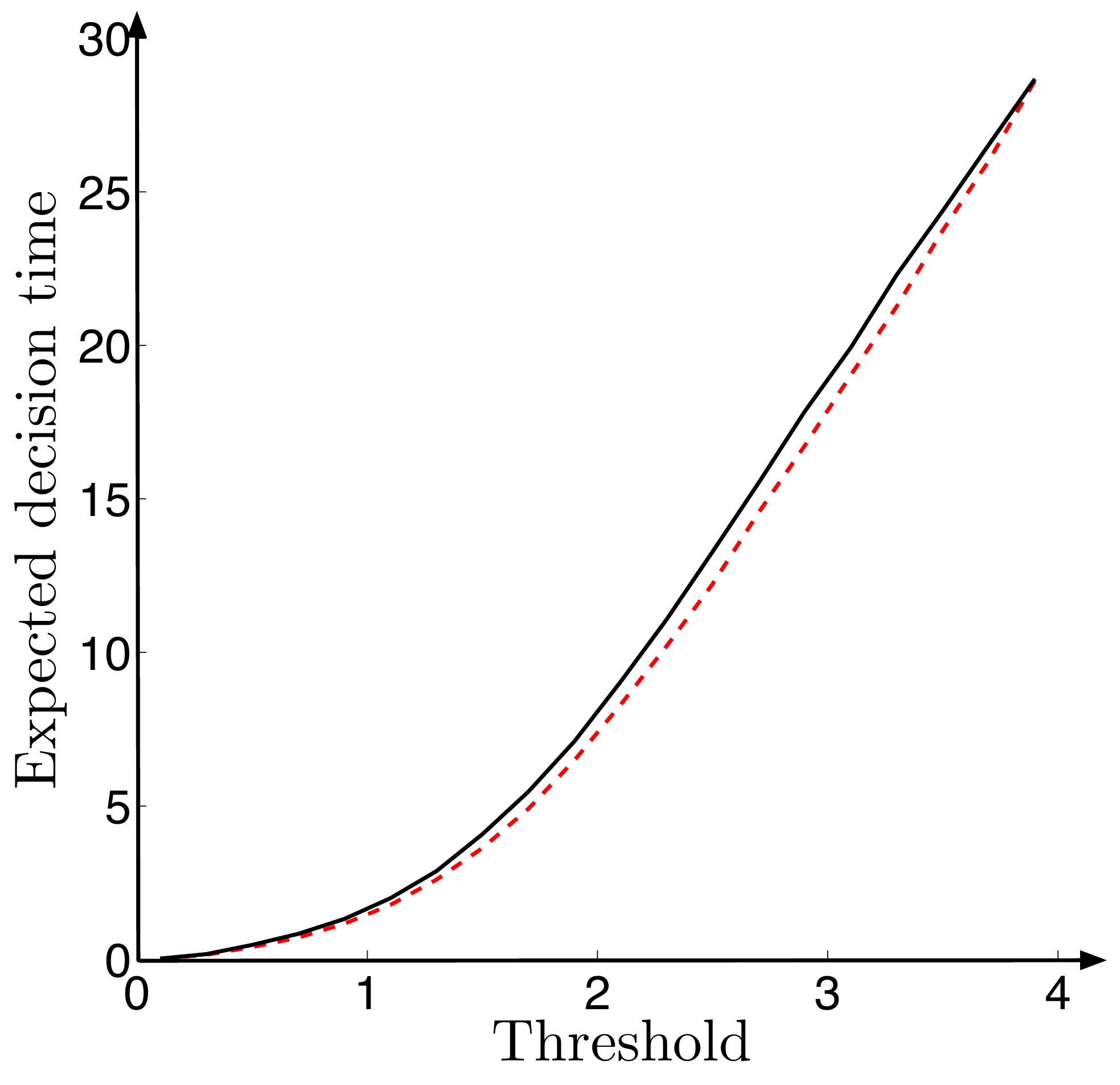}}
\subfigure[Passage time distribution]{
\includegraphics[width=0.24\textwidth]{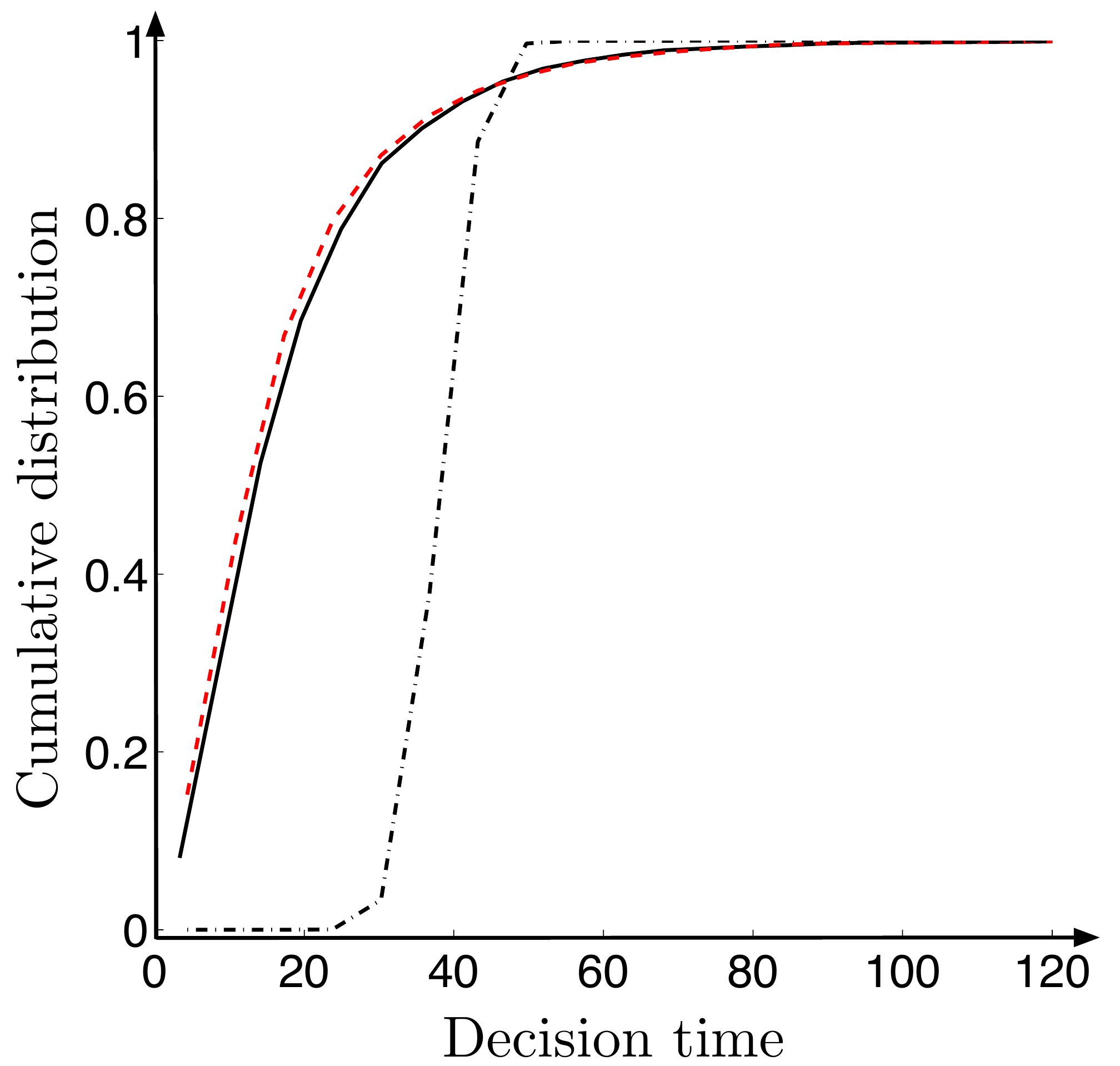} \label{fig:first-passage-comp}}
\caption{Error rates, decision times, and the first passage time distribution of the reduced DDM compared with the coupled DDM. Solid black, dashed red, and black dashed-dotted lines represent the coupled DDM, the reduced DDM, and the centralized DDM, respectively.\label{fig:ddm-reduced} {  Note that the performance of the centralized DDM, which is asymptotically equivalent to the coupled DDM, is significantly different from the performance of the coupled DDM for finite thresholds.}}
\end{figure}

We compare the error rates and decision times for the coupled DDM with the centralized DDM  in Figure~\ref{fig:coupled-DDM}. 
For the interaction topology in Figure~\ref{fig:graph} and $\beta=0.1$, we performed Monte-Carlo simulations on the coupled DDM to determine the error rates and the decision times at each node as a function of threshold value. 
\begin{figure}[ht!]
\centering 
\subfigure[Log-likelihood ratio of no error]{\includegraphics[width=0.24\textwidth]{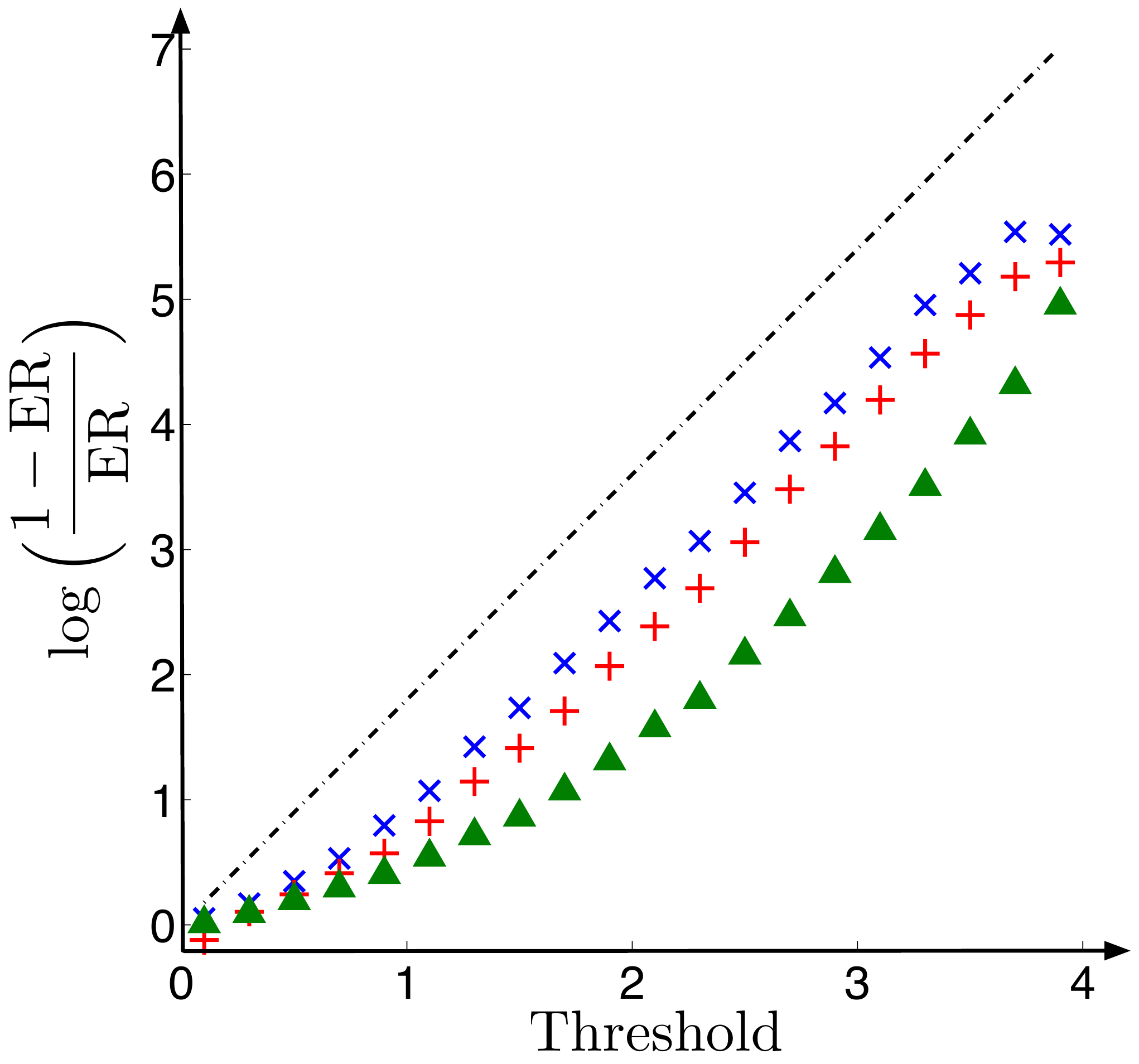}}
\subfigure[Expected decision times]{\includegraphics[width=0.24\textwidth]{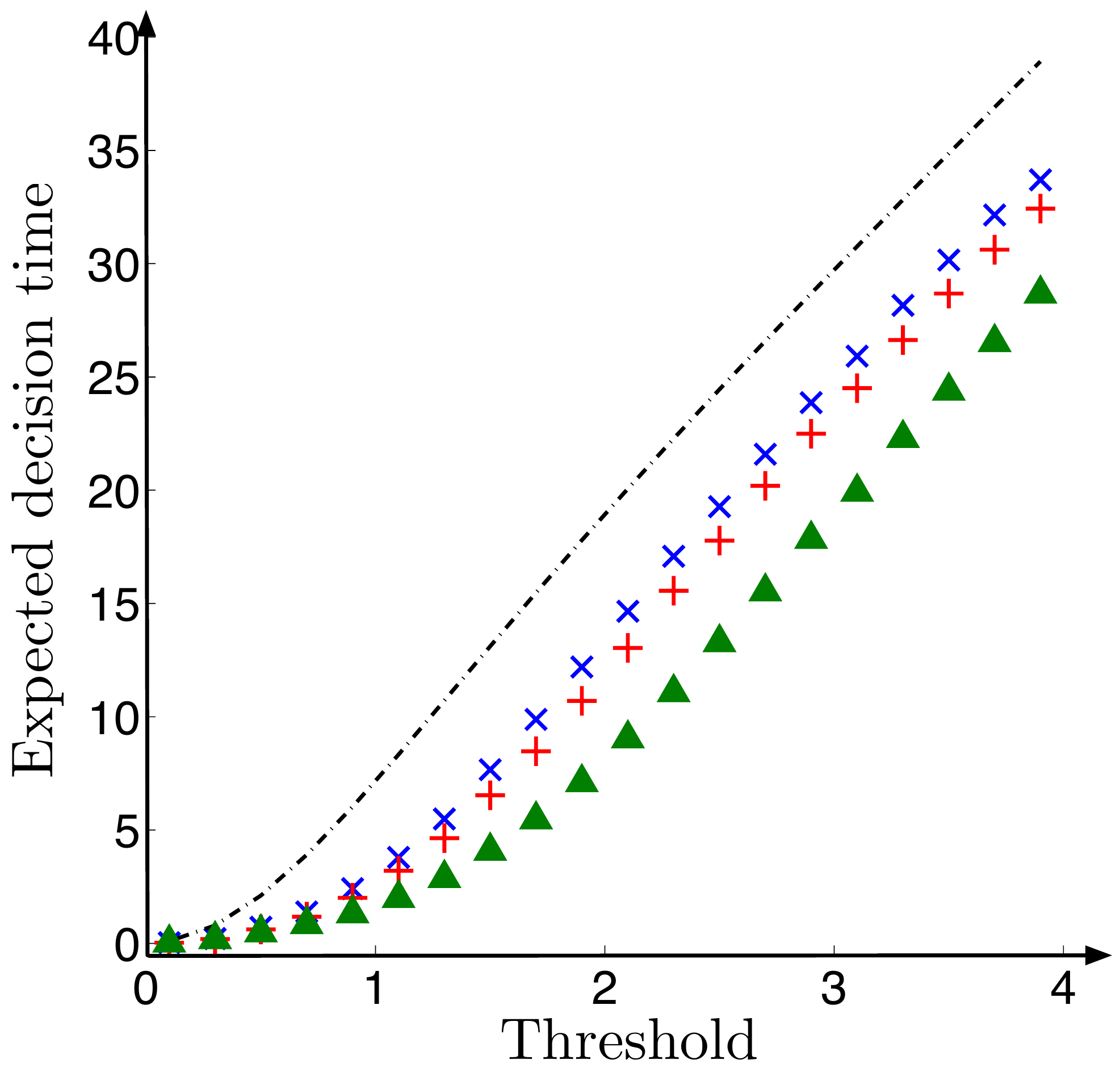}}
\caption{Comparison of the performance of the coupled DDM with the performance of the centralized DDM at each node. The dotted black line represents the performance of a centralized decision-maker. The blue $\times$, the red $+$, and the green triangles represent the performance of the coupled DDM for decision-makers $1$, $2$, and $6$, respectively.   
\label{fig:coupled-DDM}}
\end{figure}
Note that the difference in the performance of the coupled DDM and the centralized DDM is smaller for a more centrally located decision-maker. Furthermore, for large thresholds, the expected decision time graph for the coupled DDM is parallel to the  expected decision time graph for the centralized DDM. 
Thus, at large thresholds, the expected decision time graph for the coupled DDM at node $k$ can be obtained by translating the expected decision time graph for the centralized DDM horizontally to the right. Such a translation corresponds to a reduction in the threshold for the centralized DDM. 
%
%
%
%
Moreover, this reduction should be a function of the centrality of the node. This observation is in the spirit of our bounds in Proposition~\ref{prop:bounds}. In fact, insights from Proposition~\ref{prop:bounds} and these numerical results suggest that for a given instance of the coupled DDM, and large thresholds, there exists  a constant $\bar K$ such that the coupled DDM at node $k$ is equivalent to a centralized DDM with threshold $\eta_k - \bar K/ \sqrt{\mu_k}$.

We now numerically investigate the behavior of the constant $\bar K$. 
Let $\Delta T_k$ be the difference between the expected decision times at node $k$ for the centralized DDM and the coupled DDM at large thresholds. Then, the threshold for the centralized DDM should be reduced by $\beta \Delta T_k$ to capture the performance of the coupled DDM at node $k$. 
We now investigate the threshold correction $\beta \Delta T_k$ as a function of the centrality  $\mu$ of a node in the interaction graph and the drift rate. To this end, we performed Monte-Carlo simulations with Erd{\"o}s-R{\' e}yni graphs, and we plot $\beta \Delta T_k$ as a function of $1/\sqrt{\mu}$ in Figure~\ref{fig:gap}. For the Monte-Carlo simulations, we pick the number of nodes $n$ uniformly in $\{3, \ldots, 10\}$, and connect any two nodes with probability $1.1\times \log(n)/n$. 
{ We set the threshold $\eta_k$ at each node equal to $3$.}
It can be seen in Figure~\ref{fig:gap} that the threshold correction $\beta \Delta T_k$ varies linearly with $1/\sqrt{\mu}$. We further compute the the slope of the linear trend in Figure~\ref{fig:gap} as a function of the drift rate, and plot it in Figure~\ref{fig:gain}. We observe that the function $\map{\bar K}{\real_{>0}}{\real_{>0}}$ defined by
\[
\bar K(\beta) = \frac{ e^{-\frac{1}{4 \sqrt{\beta}}}}{\sqrt{\beta}(1 + \beta/3)},
\]
captures the numerically observed slope as a function of the drift rate well. The function $\bar K(\beta)$ is the red dashed curve in Figure~\ref{fig:gain}. 

\begin{figure}[ht!]
\centering 
\subfigure[Threshold correction for $\beta=0.1$]{\includegraphics[width=0.24\textwidth]{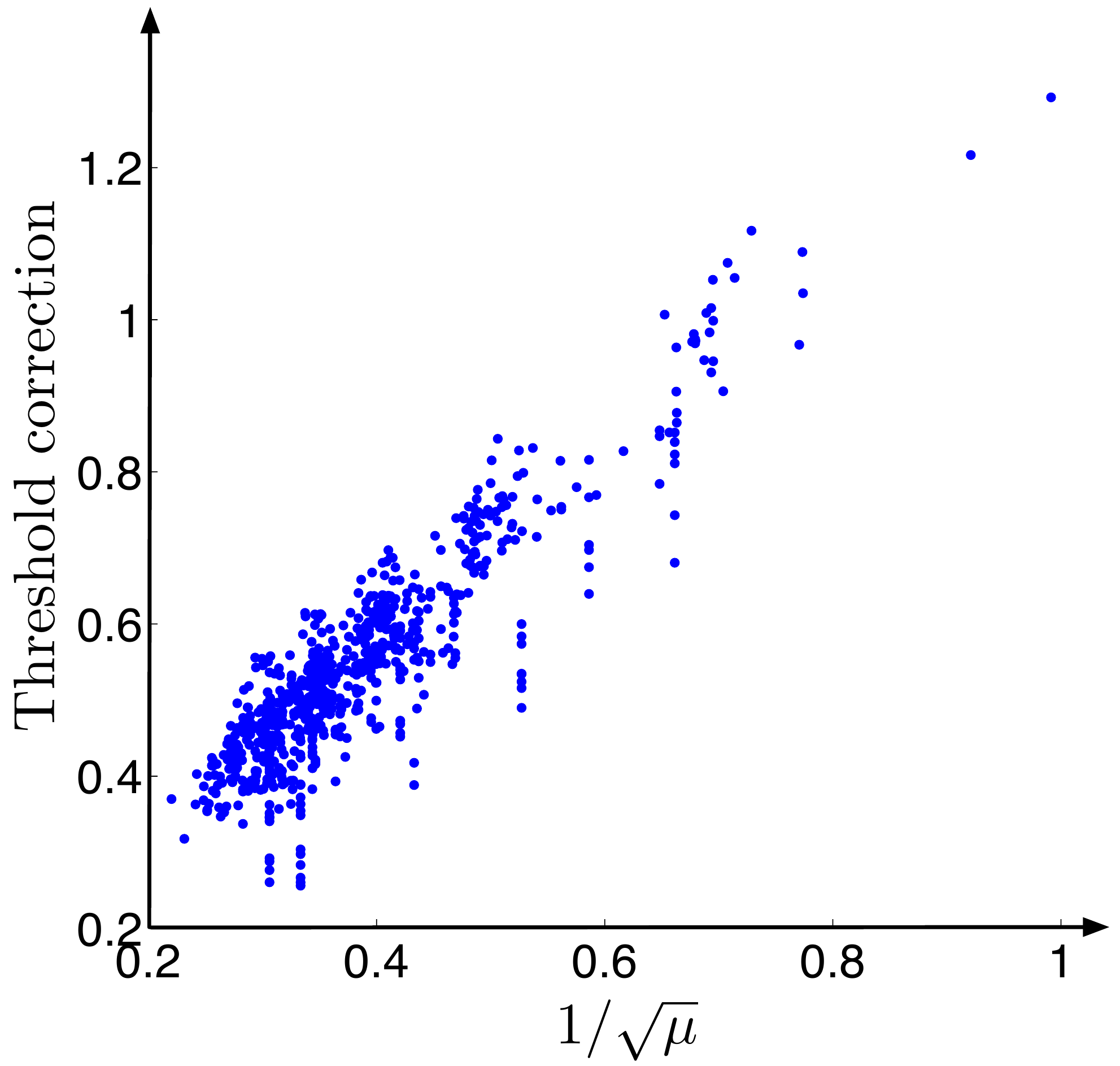}\label{fig:gap}}
\subfigure[Proportionality Constant]{\includegraphics[width=0.24\textwidth]{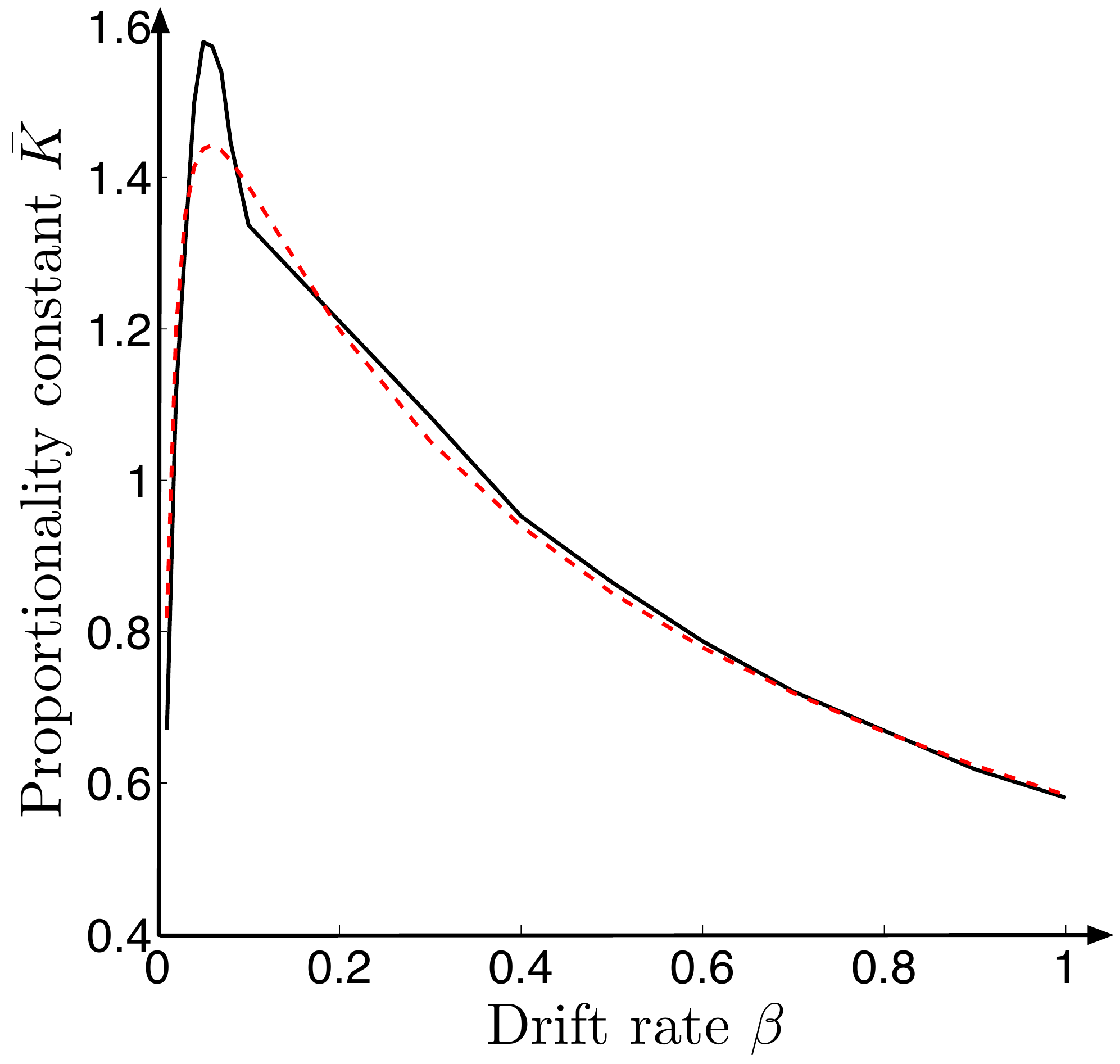}\label{fig:gain}}
\caption{(a) Threshold correction as a function of the node centrality. (b) Slope of the linear trend in (a) as a function of the drift rate $\beta$. The solid black line represents numerically computed slopes and the dashed red line represents the fitted function $\bar K(\beta)$.}
\end{figure}

We refer to the centralized DDM with threshold
$\supscr{\eta_k}{corr} =\max\{0, \eta_k - {\bar K(\beta)}/{\sqrt{\mu_k}}\}$ as the threshold corrected 
centralized DDM at node $k$. 
We compare the performance of the coupled DDM with the threshold corrected 
centralized DDM at nodes $1,2,$ and $6$ in Figure~\ref{fig:corrected-DDM}. It can be seen that the threshold corrected centralized DDM is fairly accurate at large threshold values and the minimum threshold value at which the threshold corrected centralized DDM starts to capture the performance of the coupled DDM well depends on the centrality of the node. 

\begin{figure}[ht!]
\centering 
\subfigure[Log-likelihood ratio of no error]{\includegraphics[width=0.24\textwidth]{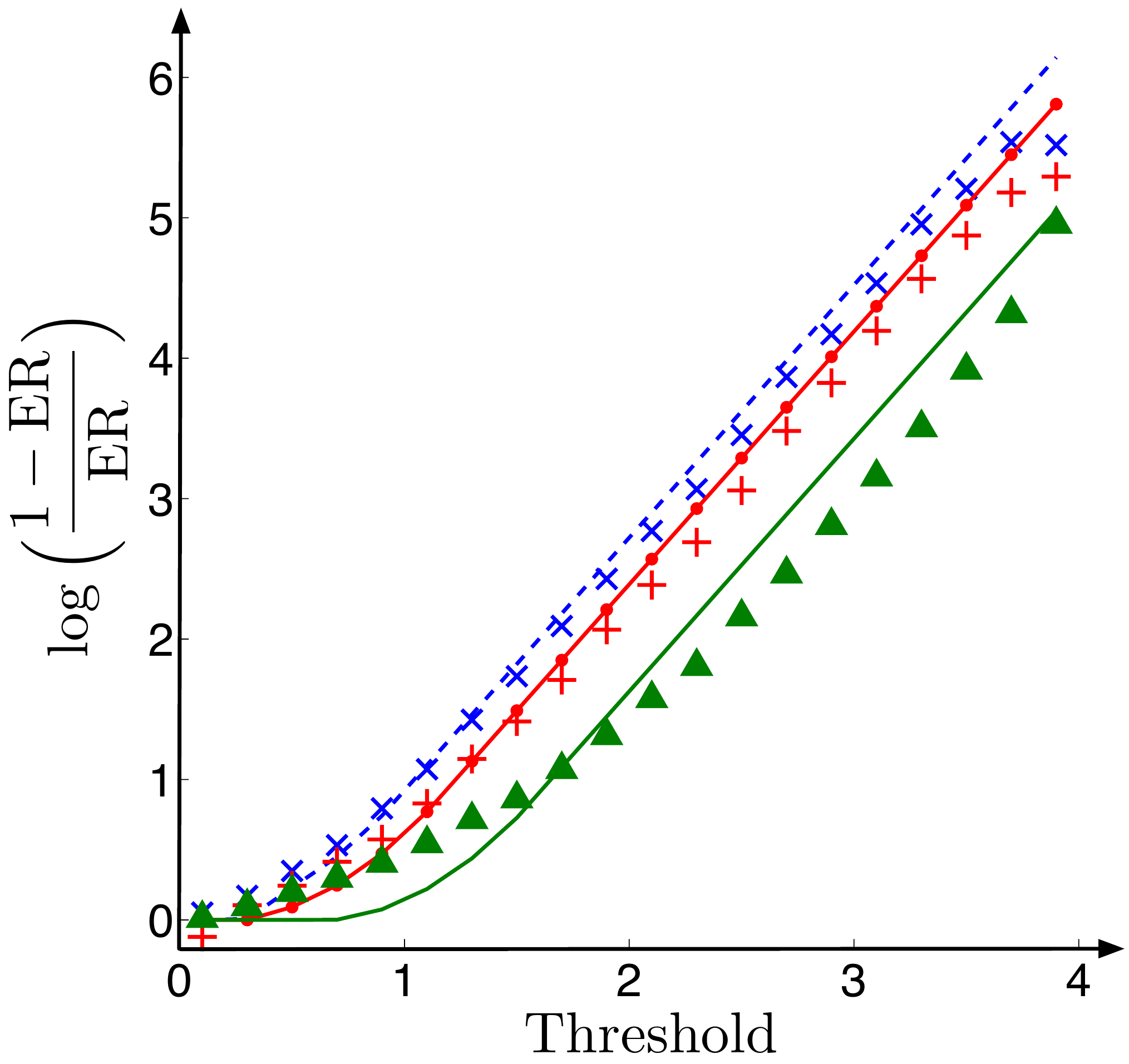}}
\subfigure[Expected decision times]{\includegraphics[width=0.24\textwidth]{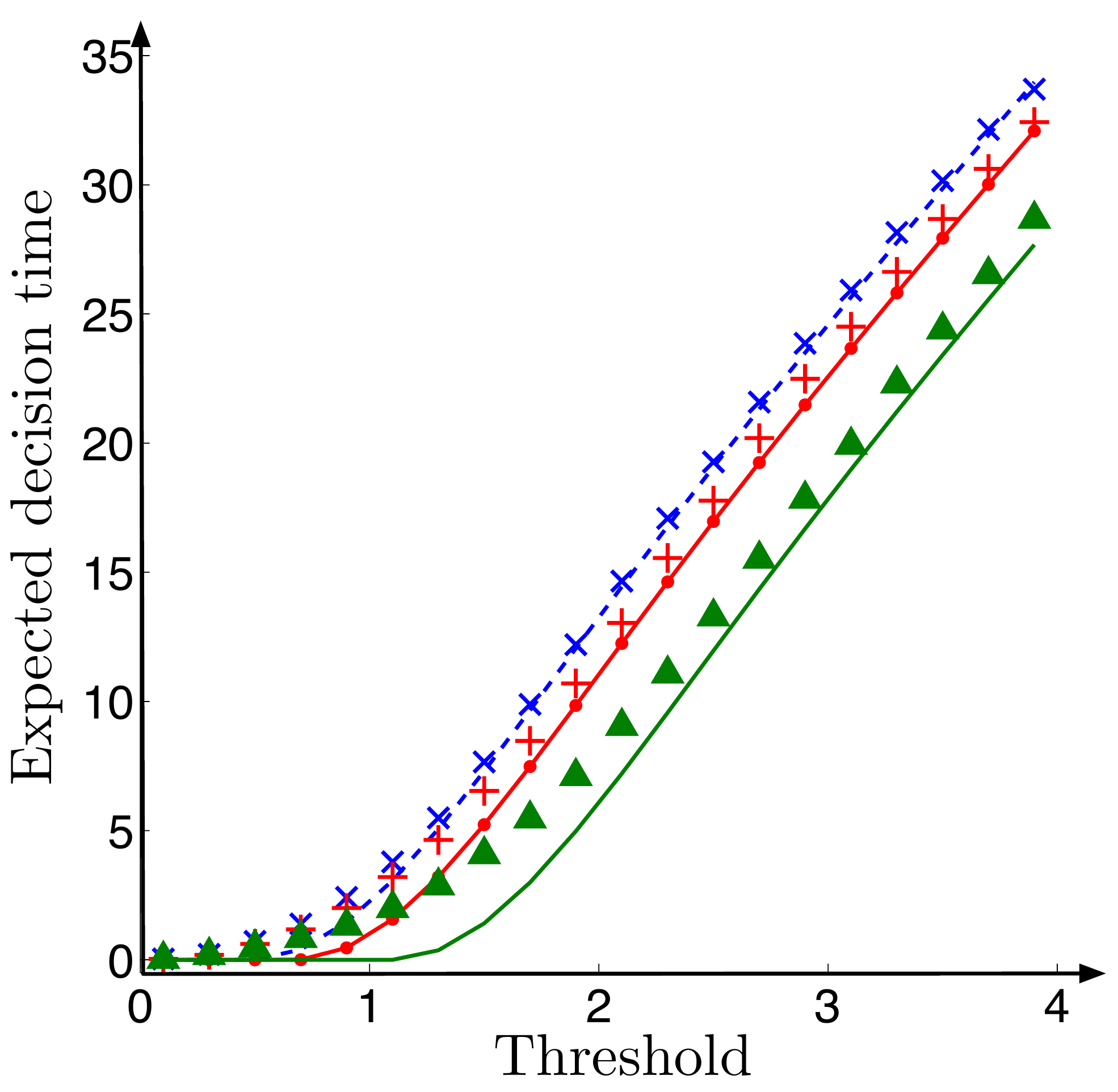}}
\caption{Comparison of the performance of the coupled DDM with the performance of the threshold corrected centralized DDM at each node. The blue $\times$, the red $+$, and the green triangles represent the performance of the coupled DDM for decision-makers $1$, $2$, and $6$, respectively.   
The blue dashed lines, the red solid lines with dots, and the green solid lines represent the performance of the threshold corrected centralized DDM for decision-makers $1$, $2$, and $6$, respectively. 
\label{fig:corrected-DDM}}
\end{figure}

\section{Optimal Threshold Design for the Speed-Accuracy Tradeoff} \label{sec:threshold-design}
In this section, we examine various threshold selection mechanisms for decision-makers in the group. We first discuss the Wald-like threshold selection mechanism that is well suited to threshold selection in engineering applications. Then, we discuss the Bayes risk minimizing mechanism and the reward rate maximizing mechanism, which are plausible threshold selection methods in human decision-making. In the following, we focus on the case of large thresholds and small error rates, and assume that the threshold correction function $\bar K$ is known. 
We also define the corrected threshold as $\supscr{\eta}{corr}_k = \max\{0,\eta_k - \bar K /\sqrt{\mu_k}\}$.

\subsection{Wald-like mechanism}
In the classical sequential hypothesis testing problem~\cite{AW:45}, the thresholds are designed such that the probability of error is below a prescribed value. In a similar spirit, we can pick threshold $\eta_k$ such that the probability of error is below a desired value $\alpha_k \in (0,1)$. 
Setting the error rate at node $k$ equal to $\alpha_k$ in the threshold corrected expression for the error rate, we obtain
\[
\supscr{\eta_k}{wald}\approx  \frac{\bar K(\beta)}{\sqrt{\mu_k}} +\frac{1}{2 \beta n} \log \Big ( \frac{1-\alpha_k}{\alpha_k} \Big).
\]

Therefore, under the Wald's criterion, if each node has to achieve the same error rate $\alpha$, then the expected decision time at node $k$ is:
\[
\expt[\supscr{T}{wald}_k] \approx \frac{1-2\alpha}{\beta} \bigg(  \frac{\bar K(\beta)}{\sqrt{\mu_k}} +\frac{1}{2 \beta n} \log \Big ( \frac{1-\alpha}{\alpha} \Big)\bigg),
\]
i.e., a more centrally located decision-maker has a smaller expected decision time. 

\subsection{Bayes risk minimizing mechanism}
The Bayes risk minimization is one of the plausible mechanisms for threshold selection for humans~\cite{RB-EB-etal:06}. In this mechanism, the threshold $\eta_k$ is selected to minimize the Bayes risk ($\mr{BR}_k$) defined by
\[
\mr{BR}_k=  c_k \mr{ER}_k +  \mr{ET}_k,
\]
where $c_k\in \real_{\ge 0}$ is a parameter determined from empirical data~\cite{RB-EB-etal:06}. 
It is known~\cite{RB-EB-etal:06} that for the centralized DDM~\eqref{eq:centralized-ddm} the threshold $\supscr{\eta}{corr}_k$ under the Bayes risk criterion is determined by the solution of the following transcendental equation: 
\begin{equation}\label{eq:bayes-risk-transcendental}
2 c_k \beta^2 n - 4 \beta n \supscr{\eta}{corr}_k + e^{-2 \beta n \supscr{\eta}{corr}_k} - e^{2 \beta n \supscr{\eta}{corr}_k} =0. 
\end{equation}
Furthermore, if the cost $c_k$ is the same for each agent, then the corrected threshold obtained from equation~\eqref{eq:bayes-risk-transcendental} is the same for each decision-maker. Consequently, the error rate and the expected decision time are the same for each agent. However, the true threshold $\eta_k$ is smaller for a more centrally located agent.

\subsection{Reward rate maximizing mechanism}
Another plausible mechanism for threshold selection in humans is reward rate maximization~\cite{RB-EB-etal:06}. The reward rate ($\mr{RR}_k$) is defined by
\[
\mr{RR}_k = \frac{1-\mr{ER}_k}{\mr{ET}_k+ \supscr{T}{motor}_k+ D_k +\mr{ER}_k D^p_k},
\]
where $\supscr{T}{motor}_k$ is the motor time associated with the decision-making process, $D_k$ is the response time, and $D^p_k$ is the additional time that decision-maker $k$ takes after an erroneous decision (see~\cite{RB-EB-etal:06} for detailed description of the parameters).  It is known~\cite{RB-EB-etal:06} that for the centralized DDM~\eqref{eq:centralized-ddm}, the threshold $\supscr{\eta}{corr}_k$ under the reward rate criterion is determined by the solution of the following transcendental equation: 
\begin{equation}\label{eq:reward-rate-threshold-transcendental}
e^{2 \beta n \supscr{\eta}{corr}_k}-1 = 2 \beta^2 n (D_k + D^p_k + \supscr{T}{motor}_k - \supscr{\eta}{corr}_k/\beta).
\end{equation}
Moreover, if the parameters $\supscr{T}{motor}_k, D_k$, and $D^p_k$ are the same for each agent, then the corrected threshold $\supscr{\eta}{corr}_k$ obtained from equation~\eqref{eq:reward-rate-threshold-transcendental} is the same for each decision-maker. Consequently, the error rate and the expected decision time are the same for each agent. However, the true threshold $\eta_k$ is smaller for a more centrally located agent.

We now summarize the effect of the node centrality on the performance of the reduced DDM under four
threshold selection criteria, namely, (i) fixed threshold at each node, (ii) Wald criterion, (iii) Bayes risk, and (iv) reward rate in Table~\ref{table:performance-centrality}.

\begin{table}[ht!]
\caption{Behavior of  the performance  with increasing node centrality.\label{table:performance-centrality}}
\resizebox{\linewidth}{!} {
\begin{tabular}{|c||c|c|c|c|c|}
\hline
& Error rate & Expected  & Threshold & Bayes risk & Reward rate \\
& &decision time & & & \\
\hline\hline
Fixed threshold & decreases & increases & constant & - & - \\
\hline
Wald & constant & decreases & decreases & decreases & increases\\
\hline
Bayes risk & constant & constant & decreases & constant & constant\\
\hline
Reward rate & constant & constant & decreases & constant & constant \\
\hline\hline
\end{tabular}}
\end{table}

\section{Extensions to Other Decision-making Models}\label{sec:other-models}
In this section we extend the coupled DDM to other decision-making models. We first present the Ornstein-Uhlenbeck (O-U) model for human decision-making in two alternative choice tasks with recency effects, and extend the coupled DDM  to the coupled O-U model. We then present the race model for human decision-making in multiple alternative choice tasks, and extend the coupled DDM to the coupled race model.  

\subsection{Ornstein-Uhlenbeck model}
The DDM is an ideal evidence aggregation model and assumes a perfect integration of the evidence. However, 
in reality, the evidence aggregation process has recency effects, i.e., the evidence aggregated later has more influence in the decision-making than the evidence aggregated earlier.
The Ornstein-Uhlenbeck (O-U) model extends the DDM for human decision-making to incorporate recency effects and is described as follows:
\begin{equation}\label{eq:o-u-process}
\mr d x(t) = (\beta -\theta x(t)) \mr d t + \sigma \mr d W(t), \quad x(0)=0, 
\end{equation}
where $\theta \in \real_{\ge 0}$ is a constant that produces a decay effect over the evidence aggregation process~\cite{RB-EB-etal:06}.  { It can be seen using the Euler discretization of~\eqref{eq:o-u-process} that the O-U model is the continuum limit of an autoregressive (AR(1)) model, and assigns exponentially decreasing weights to past observations.}

The evidence aggregation process~\eqref{eq:o-u-process}  is Markovian, stationary and Gaussian. The mean and the variance of the evidence $x(t)$ at time $t$ are $\beta (1-e^{-\theta t})/\theta$ and $\sigma^2(1-e^{-2\theta t})/2\theta$, respectively. 
{ The two decision hypotheses correspond to the drift rate being positive or negative, i.e., $\beta \in\real_{>0}$ and $\beta \in \real_{<0}$, respectively.}
The decision rules for the O-U model are the same as the decision rules for the DDM. 
The expected decision time and the error rate for the O-U model can be characterized in closed form. We refer the reader to~\cite{RB-EB-etal:06} for details.

\subsection{The coupled O-U model}
We now extend the coupled DDM to the coupled O-U model. In the spirit of the coupled DDM, we model the evidence aggregation across the network through the Laplacian flow. Without loss of generality,  we assume the diffusion rate is unity. The coupled O-U model is described as follows:
\begin{equation*}
\mr d \bs x(t) = (-(L + \theta \mc I_n) \bs x(t) + \beta \bs 1_n ) \mr dt + \mr d \bs W_n(t), \; \bs x(0)=\bs 0_n,
\end{equation*}
where $\bs x \in \real^n$ is the vector of evidence for each agent, $L \in \real^{n\times n}$ is the Laplacian matrix associated with the interaction graph, $\theta \in \real_{>0}$ is a constant, $\beta \in \real$ is the drift rate and $\bs W_n(t)$ is the standard $n$-dimensional Weiner process.  

Similar to the coupled DDM, it can be shown that the solution to the coupled O-U model is a Gaussian process with mean and covariance at time $t$ given by
\begin{align*}
\expt[\bs x(t)] &= \frac{\beta(1-e^{-\theta t})}{\theta} \bs 1_n, \; \text{and}\\
\text{Cov}(x_k(t), x_j(t)) &= \sum_{p=1}^n \frac{1-e^{-2(\lambda_p +\theta) t}}{2(\lambda_p +\theta)} u_k^{(p)} u_j^{(p)},
\end{align*}
where $\lambda_p, p\in\until{n}$ are the eigenvalues of the Laplacian matrix $L$ and $\bs u^{(p)}$ are the associated eigenvectors.

Similar to Section~\ref{sec:prop-coupled-ddm},  principle component analysis followed by the error approximations yield the following reduced O-U model as a decoupled approximation to the coupled O-U model at node $k$:
\begin{equation*}
\begin{bmatrix} \mr d  y_k(t) \\ \mr d  \varepsilon_k(t) \end{bmatrix}
= \left[\begin{smallmatrix}\beta -\theta y_k(t) + (\theta -\frac{\hat u_k}{2}) \varepsilon_k(t) \\ -\frac{\hat \mu_k  }{2} \varepsilon_k(t)  \end{smallmatrix}\right] \mr d t 
+  \begin{bmatrix} \frac{1}{\sqrt{n}} & 1 \\ 0 & 1 \end{bmatrix} \begin{bmatrix} \mr d  W_1(t) \\ \mr d  W_2(t) \end{bmatrix},
\end{equation*}
where $y_k$ is the evidence aggregated at node $k$, $\varepsilon_k$ is the error defined analogous to the error in the reduced DDM, and {$1/\hat \mu_k = \sum_{p=2}^n \frac{1}{2(\lambda_p +\theta)} (u_k^{(p)})^2$. }

Furthermore, similar to Proposition~\ref{prop:er-dt-pdes},  partial differential equations to compute the expected decision time and the error rate for the reduced O-U process can be derived. For parameter values in Section~\ref{sec:numerical} and $\theta=0.1$, a comparison between the performance of the coupled O-U model and the reduced O-U model is presented in Figure~\ref{fig:coupled-o-u}. 

\begin{figure}
\subfigure[Log-likelihood ratio of no error]{
\includegraphics[width=0.23\textwidth]{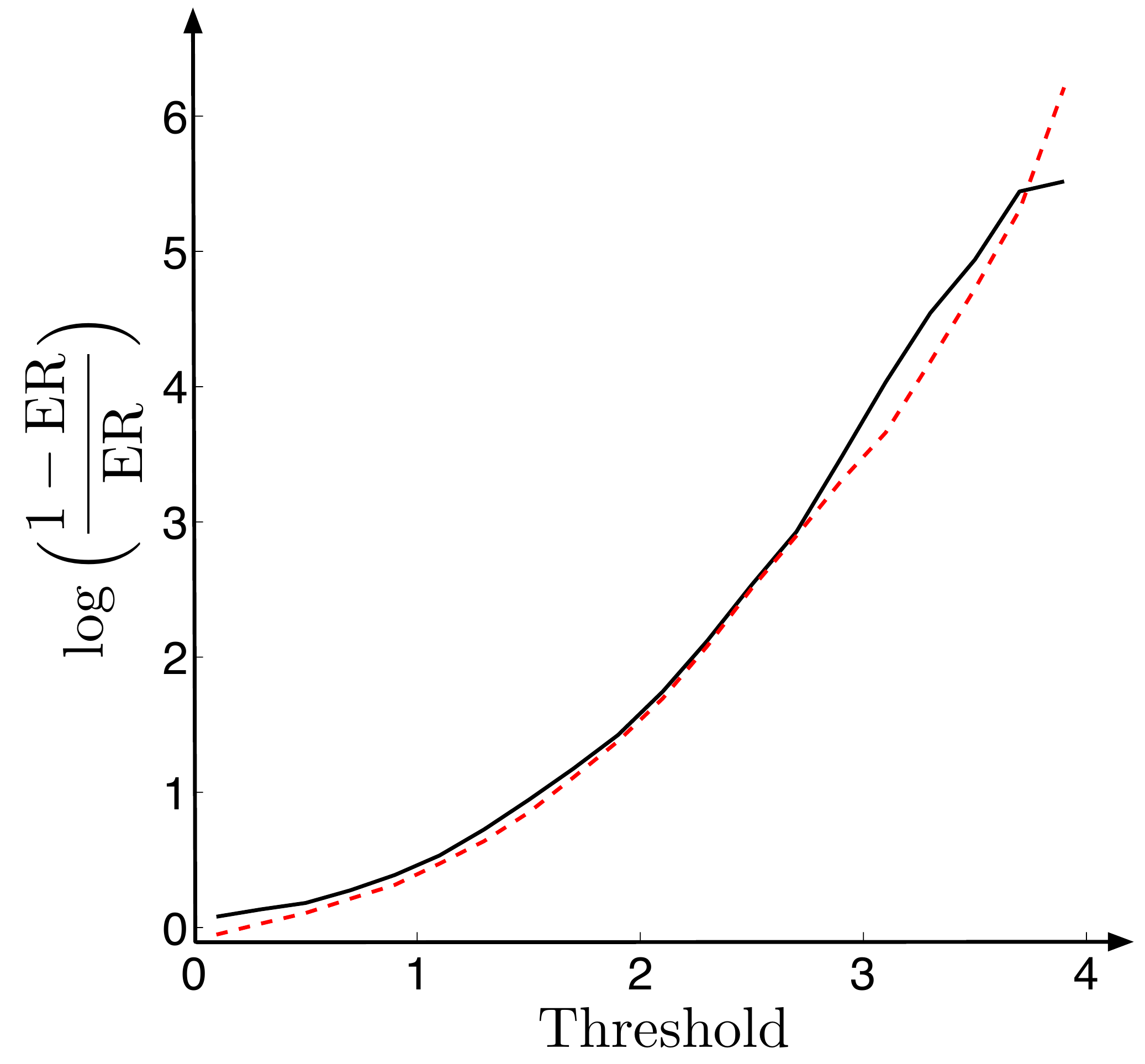}}
\subfigure[Expected decision times]{
\includegraphics[width=0.23\textwidth]{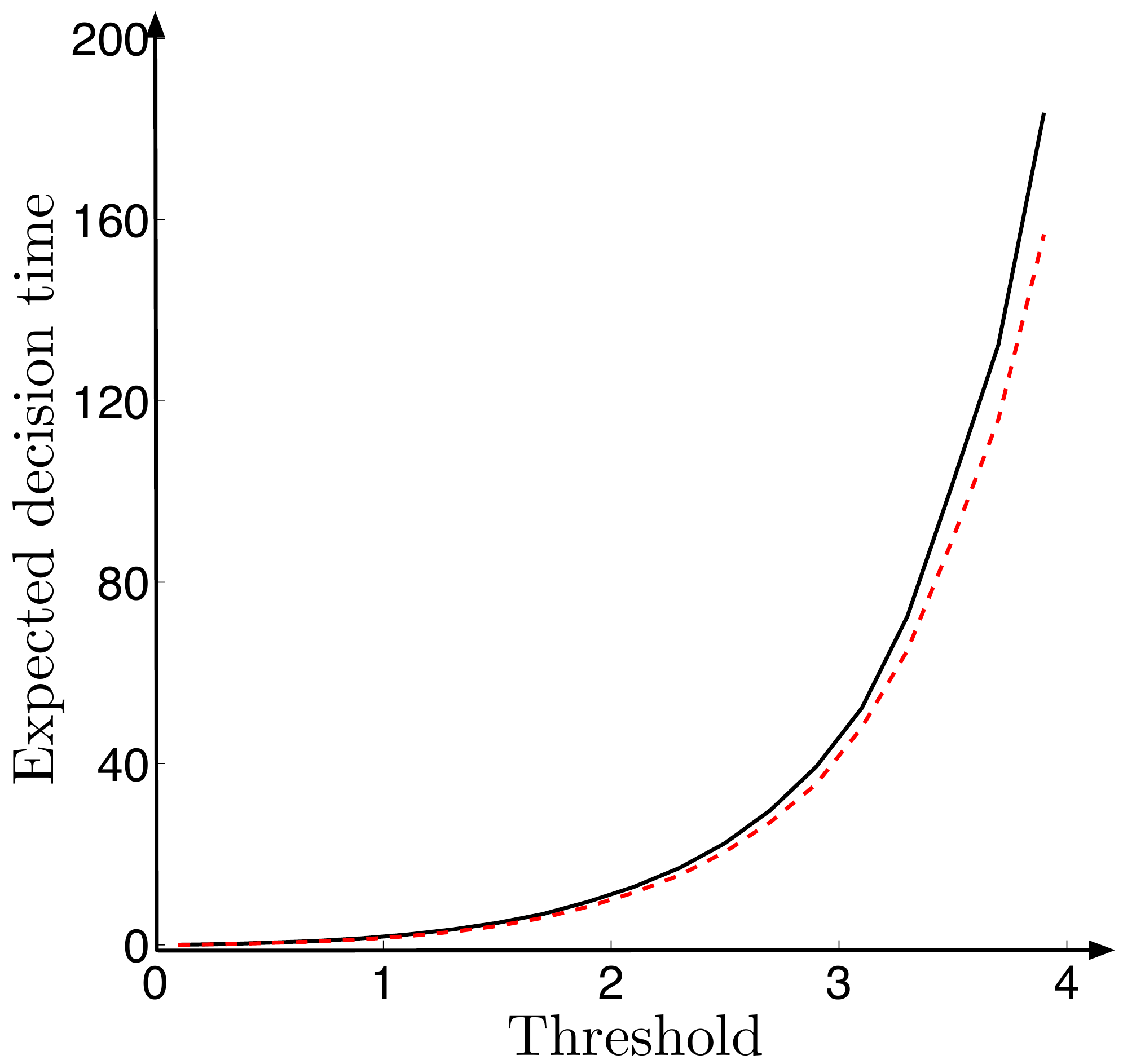}}
\caption{Error rates and decision times for reduced O-U model compared with the coupled O-U model. Solid black and dashed red lines represent the coupled O-U model and the reduced O-U model, respectively. \label{fig:coupled-o-u}}
\end{figure}

\subsection{The race model}
Consider the decision-making scenario in which the decision-maker has to choose amongst $m$ possible alternatives. Human decision-making in such multi-alternative choice tasks is modeled well by the \emph{race model}~\cite{TM-PH:06} described below.

Let the evidence aggregation process for an alternative $a\in \until{m}$ be modeled by the DDM
\begin{equation}\label{eq:evidence-multiple-alternative}
\mr dx^a(t) = \beta^a \mr dt + \sigma \mr d W^a(t),
\end{equation}
where $x^a(t)$ is the evidence in favor of alternative $a$ at time $t$, $\beta^a$ is the drift rate, $\sigma \in \real_{>0}$ is the diffusion rate, $W^a(t)$ is the realization of the standard one-dimensional Weiner process for alternative $a$. { The decision hypotheses correspond to the drift rate being positive for one alternative and zero for every other alternative, i.e., $\beta^{a_0} \in \real_{>0}$, for some  $a_0 \in \until{m}$, and $\beta^a =0$, for each $a\in \until{m}\setminus \{a_0\}$.}

For the evidence aggregation process~\eqref{eq:evidence-multiple-alternative} and the free response paradigm, the decision is made in favor of the first alternative $a\in \until{m}$ that satisfies
\begin{equation}\label{eq:decision-rule}
x^a(t) - \max \setdef{x^j(t)}{j\in\until{m}\setminus\{a\}} \ge  \eta^a,
\end{equation}
where $\eta^a$ is the threshold for alternative $a$. For a prescribed maximum probability $\mr{R}^a$ of incorrectly deciding in favor of alternative $a$, the threshold is selected as $\eta^a = \log ((m-1)/m\mr{R}^a)$.

For the race model~\eqref{eq:evidence-multiple-alternative} and the decision rule~\eqref{eq:decision-rule}, the mean reaction time and the error rate can be asymptotically characterized; see~\cite{TM-PH:06} for details. 
The race model is the continuum limit of  an asymptotically optimal sequential multiple hypothesis test proposed in~\cite{VPD-AGT-VVV:99}. 

%

\subsection{The coupled race model}
We now develop a distributed version of the race model~\eqref{eq:evidence-multiple-alternative}. Without loss of generality,  we assume the diffusion rate is unity. In the spirit of the coupled DDM, we use the Laplacian flow to aggregate the evidence across the network. 
Let the evidence in favor of alternative $a$ at node $k$ and at time $t$ be $x_k^a(t)$.
Let $\vec {x}_k (t) \in \real^m$ 
be the column vector  with entries $x^a_k(t), a\in \until{m}$ 
and $\bs {\vec x}(t) \in \real^{mn}$ be the column vector formed by concatenating vectors 
$\vec x_k (t) \in \real^m$. We define the coupled race model by 
\begin{equation}\label{eq:coupled-race}
\mr d \bs {\vec x}(t) = -(L \otimes I_m) \bs {\vec x}(t) \mr dt +( \bs 1_n \otimes \bs \beta ) \mr dt + \mr d \bs{\vec{W}}_{mn}(t),
\end{equation}
with initial condition $\vec{\bs x}(0) = \bs 0_{mn}$, where $\otimes$ denotes the Kronecker product, $\bs \beta \in \real^m$ is the column vector with entries $\beta^a, a\in \until{m}$, and  $\bs{\vec{W}}_{mn}(t)$ is the standard $mn$-dimensional Weiner process.
Note that dynamics~\eqref{eq:coupled-race} are equivalent to running a set of $m$ parallel coupled DDMs, one for each alternative. 

For the evidence aggregation process~\eqref{eq:coupled-race}, node $k$ makes a decision in favor of the first alternative $a\in \until{m}$ that satisfies
\begin{equation*}
x^a_k(t) - \max \setdef{x^j_k(t)}{j\in\until{m}\setminus\{a\}} \ge \eta^a_k,
\end{equation*}
where $\eta^a_k $ is the threshold for alternative $a$ at node $k$.

We define the \emph{centralized race model} as the race model in which at each time all the evidence distributed across the network is available at a fusion center. Such a centralized DDM is obtained by replacing $\sigma$ in  ~\eqref{eq:evidence-multiple-alternative}  with
$1/\sqrt{n}$. It can be shown along the lines of Proposition~\ref{prop:asymptotic-optimality} that the coupled race model is asymptotically equivalent to the centralized race model, and hence, is asymptotically optimal. 


As pointed out earlier, the coupled race model is equivalent to a set of $m$ parallel coupled DDMs. Thus, the analysis for coupled DDM extends to the coupled race model in a straightforward fashion. In particular, for each alternative, the evidence aggregation process can be split into the centralized process and the error process, which can be utilized to construct reduced DDMs for each alternative. Furthermore, similar to the case of the coupled DDM, threshold corrections can be computed for the coupled race model.


\section{Conclusions and Future Directions}\label{sec:conclusions}
In this paper, we used the context of two alternative choice tasks to study the speed-accuracy tradeoff in collective decision-making for a model of human groups with network interactions. 
We focused on the free response decision-making paradigm in which each individual takes their time to make a decision. We utilized the Laplacian flow based coupled DDM to capture the evidence aggregation process across the network. We developed the reduced DDM, a decoupled approximation to the coupled DDM. We characterized the efficiency of the decoupled approximation and derived partial differential equations for the expected decision time and error rate for the reduced DDM. We then derived high probability bounds on the expected decision time and error rate for the reduced DDM. 
We characterized the effect of the node centrality in the interaction graph of the group on  decision-making metrics under several threshold selection criteria. 
Finally, we extended the coupled DDM to the coupled O-U model for decision-making in the two alternative choice task with recency effects, and the coupled race model for the multi-alternative choice task.

There are several possible extensions to this work. First, in this paper, we utilized the Laplacian flow to model the evidence aggregation process across the network. It is of interest to consider other communication models for evidence aggregation across network, e.g., gossip communication, bounded confidence based communication, etc. Second, in this paper, we assumed that the drift rate for each agent is the same. However, in the context of robotic groups or animal groups, it may be the case that only a set of individuals (leaders)  have a positive drift rate while others individuals (followers) may have zero drift rate. It is of interest to extend this work to such leader-follower networks. Third, it is of interest to extend the results in this paper to more general decision-making tasks, e.g., the multi-armed bandit tasks~\cite{PR-VS-NEL:13d}.

\section*{Acknowledgments}
The authors thank  Philip Holmes and Jonathan D. Cohen for helpful discussions. We also thank the editor and two anonymous referees, whose comments greatly strengthened and clarified the paper.

\end{document}